\documentclass{amsart}
\usepackage{amsmath,amssymb,amscd,latexsym}
\usepackage{enumerate,verbatim,calc}
\usepackage{rotate}
\usepackage[all]{xy}
\SelectTips{eu}{}
\xyoption{curve}
\CompileMatrices

\usepackage{graphics}
\usepackage[colorlinks]{hyperref}

\newcommand{\ts}{\textstyle}

\newcommand{\bwedge}{{\textstyle\bigwedge}}
\newcommand{\coker}{\operatorname{Coker}}
\newcommand{\fm}{{\mathfrak m}}
\newcommand{\fn}{{\mathfrak n}}
\newcommand{\fp}{{\mathfrak p}}
\newcommand{\cls}[1]{{\operatorname{cls}(#1)}}
\newcommand{\col}{\colon}
\newcommand{\dd}{\partial}
\newcommand{\fd}{\operatorname{fd}}
\newcommand{\hch}[4]{\operatorname{HH}_{#1}(#3\var #2;#4)}
\newcommand{\hcoh}[4]{\operatorname{HH}^{#1}(#3\var #2;#4)}

\newcommand{\env}[2][{}]{{#2}{}^{\mathsf e}_{#1}}
\newcommand{\envv}[2]{{#1}_{#2}^{\mathsf e}}
\newcommand{\hh}[1]{\operatorname{H}(#1)}
\newcommand{\HH}[2]{\operatorname{H}_{#1}(#2)}
\newcommand{\CH}[2]{\operatorname{H}^{#1}(#2)}

\newcommand{\hra}{\hookrightarrow}
\newcommand{\id}{\operatorname{id}}
\newcommand{\Ker}{\operatorname{Ker}}
\newcommand{\lra}{\longrightarrow}
\newcommand{\op}[1]{{#1}{}^{\mathsf o}}

\newcommand{\sfc}{\mathsf c}

\newcommand{\projres}{{\mathsf p}}
\newcommand{\injres}{{\mathsf i}}

\newcommand{\nat}{{}^\natural}

\newcommand{\ov}{\bar}
\newcommand{\ul}{\underline}
\newcommand{\pd}{\operatorname{pd}}
\newcommand{\Shift}{\mathsf{\Sigma}}
\newcommand{\Spec}{\operatorname{Spec}}

\newcommand{\Mor}[3]{\operatorname{Mor}_{#1}(#2,#3)}
\newcommand{\Ext}[4]{\operatorname{Ext}^{#1}_{#2}(#3,#4){}}
\newcommand{\Hom}[3]{\operatorname{Hom}_{#1}(#2,#3)}
\newcommand{\Rhom}[3]{\operatorname{\mathsf{R}Hom}_{#1}(#2,#3)}
\newcommand{\dtensor}[1]{\otimes^{\mathsf{L}}_{#1}}
\newcommand{\Tor}[4]{\operatorname{Tor}_{#1}^{#2}(#3,#4){}}
\newcommand{\dcat}[1][S]{{\mathsf D}(#1)}
\newcommand{\dcatf}[1]{{\mathsf P}(#1)}

\newcommand{\tra}{\twoheadrightarrow}
\newcommand{\var}{{\hskip.7pt\vert\hskip.7pt}}

\newcommand{\wt}{\widetilde}
\newcommand{\xla}{\xleftarrow}
\newcommand{\xra}{\xrightarrow}

\newcommand{\BN}{{\mathbb N}}
\newcommand{\BZ}{{\mathbb Z}}

\theoremstyle{plain}

\newtheorem{theorem}{Theorem}[section]

\newtheorem{lemma}[theorem]{Lemma}
\newtheorem{corollary}[theorem]{Corollary}

\newtheorem{itheorem}{Theorem}

\newtheorem{mysublemma}{Lemma}[theorem]

\theoremstyle{definition}

\newtheorem{example}[theorem]{Example}

\newtheorem{construction}[theorem]{Construction}

\newtheorem{chunk}[theorem]{}

\newtheorem{subchunk}[mysublemma]{}

\theoremstyle{remark}

\newtheorem{step}{Step}
\newtheorem*{Notation}{Notation}
\newtheorem{remark}[theorem]{Remark}
\newtheorem{remarks}[theorem]{Remarks}

\numberwithin{equation}{theorem}

\font\cms=cmss10
\font\cmss=cmss8
\font\cmt=cmtex10

\newcommand{\D}{{\text {\cms\char'104}}}
\newcommand{\qc}{{\text {\cmss\char'161}{{\text {\cmss\char'143}}}}}
\newcommand{\Dqc}{\D_\qc\mkern1mu}
\newcommand\Dqcpl{\D_\qc^{\lift.95,\text{\cmt\char'053},}}
\newcommand{\Dc}{\D_{{{\text {\cmss\char'143}}}}\mkern1mu}
\newcommand\Dcmi{\D_{{{\text {\cmss\char'143}}}}^{\lift.95,\text
{\cmt\char'055},}}

\newcommand{\Otimes}[1]{\otimes_{#1}^{\mathsf L}}
\newcommand\iso{{\mkern8mu\longrightarrow \mkern-25.5mu{}^\sim\mkern17mu}}
\newcommand\tiso{{\mkern6mu\rightarrow\mkern-21mu{\vphantom{h}}^\simeq\mkern9mu}}
\newcommand{\<}{\mkern-1mu}
\renewcommand{\>}{\mkern1mu}

\newcommand{\OX}{{\mathcal O_{\<\<X}}}
\newcommand{\bL}{\mathsf L}

\newcommand{\R}{\mathsf R}
\newcommand{\RH}{\mathsf{R}\mathcal{H}om}
\newcommand{\Rf}{\R f_{\<*}}

\def\lift#1,#2,{\vbox to 0pt{\vskip-#1 ex\hbox{$\scriptstyle #2$}\vss}}

\begin{document}

\title[Reduction of derived Hochschild functors]
{Reduction of derived Hochschild functors\\
over commutative algebras and schemes}

\author[L.\,L.\,Avramov]{Luchezar L.~Avramov} \address{Department of Mathematics, 
University of Nebraska,  Lincoln, NE 68588, U.S.A.}  
\email {avramov@math.unl.edu}

\author[S.\,B.\,Iyengar]{Srikanth B.~Iyengar} \address{Department of Mathematics, 
University of Nebraska, Lincoln, NE 68588, U.S.A.}  
\email{iyengar@math.unl.edu}

\author[J.\,Lipman]{Joseph Lipman} \address{Department of Mathematics, 
Purdue University, W. Lafayette, IN 47907, U.S.A.} 
\email{jlipman@purdue.edu}

\author[S.\,Nayak]{Suresh Nayak} \address{Chennai Mathematical
Institute, Siruseri 603103, INDIA}  
\email{snayak@cmi.ac.in}

\thanks{Research partly supported by NSF grants
DMS 0201904 and DMS 0803082 (LLA), 
DMS 0602498 (SBI), and NSA grant H98230-06-1-0010 (JL)} 



\date{\today}

\begin{abstract}
We study  functors underlying derived Hochschild cohomology,
also called Shukla cohomology, of a commutative algebra $S$
essentially of finite\- type and of finite flat dimension
over a commutative noetherian ring~$K$.  We~construct a
complex of $S$-modules $D$, and natural reduction isomorphisms
$\mathrm{Ext}^{*}_{S\otimes^{\mathsf L}_KS}(S\var K;M\otimes_{K}^{\mathsf
L}N) \simeq\mathrm{Ext}^{*}_{S}(\mathrm{\mathsf{R}Hom}_{S}(M,D),N)$
for all complexes of \mbox{$S$-modules $N$} and all complexes $M$
of finite flat dimension\- over~$K$ whose homology~$\mathrm{H}(M)$ is
finitely generated over $S$; such isomorphisms determine~$D$ up to derived
isomorphism.  Using Grothendieck duality theory we establish analogous
isomorphisms for any  essentially finite-type flat map $f\colon X\to Y$ of
noetherian schemes, with $f^!\mathcal O_Y$ in place of~$D$.\looseness=-1
 \end{abstract}

\maketitle

\setcounter{tocdepth}{1}
\tableofcontents

   \section*{Introduction}
We study commutative algebras essentially of finite type over some
commutative noetherian ring $K$.  Let  $\sigma\col K\to S$ denote the
structure map of such an algebra. When $S$ is projective as a $K$-module,
for example, when $K$ is a field, the Hochschild cohomology $\hcoh*KS-$
allows one to investigate certain properties of the homomorphism $\sigma$
in terms of properties of $S$, viewed as a module over the enveloping
algebra $\env S=S\otimes_{K}S$.  This comes about via isomorphisms
 \[
\hcoh nKSL = \Ext n{\env S}SL\,,
 \]
established by Cartan and Eilenberg~\cite{CE} for an arbitrary
$S$-bimodule $L$.

In the absence of projectivity, one can turn to a cohomology theory
introduced by MacLane \cite{Ml0} for $K=\BZ$, extended by Shukla~\cite{Sh}
to all rings $K$, and recognized by Quillen~\cite{Qu1} as a derived
version of Hochschild cohomology; see Section \ref{Derived Hochschild
functors over algebras}.

A central result of this article is a reduction of the computation of
derived Hochschild cohomology with coefficients in $M\dtensor{K}N$ to
a computation of iterated derived functors over the ring $S$ itself;
this is new even in the classical situation.

We write $\dcat$ for the derived category of $S$-modules, and
$\dcatf{\sigma}$ for its full subcategory consisting of complexes with
finite homology that are isomorphic in $\dcat[K]$ to bounded complexes
of flat $K$-modules.  As part of Theorem~\ref{thm:reduction2} we prove:

\begin{itheorem}
\label{ithm:main}
When $S$ has finite flat dimension as a $K$-module there exists a 
unique up to isomorphism complex $D^{\sigma}\in\dcatf\sigma$, such 
that for each $M\in\dcatf{\sigma}$ and every $N\in\dcat$ there is an
isomorphism that is natural in $M$ and $N$:
 \[
\Rhom {S\dtensor{K}S}S{M\dtensor{K}N} \simeq \Rhom S{\Rhom
SM{D^{\sigma}}}N\,.
 \]
  \end{itheorem}

The complex $D^\sigma$ is an algebraic version of a \emph{relative
dualizing complex} used in algebraic geometry, see~\eqref{reldual}.
A direct, explicit construction of~$D^{\sigma}$ is given in
Section~\ref{Relative dualizing complexes}.  When $S$ is flat as a
$K$-module, $M$ and $N$ are $S$-modules, and $M$ is flat over $K$ and
finite over $S$, the theorem yields isomorphisms of $S$-modules
 \[
\Ext n{\env S}S{M\otimes_KN}\cong\Ext{n}{S}{\Rhom SM{D^{\sigma}}}N
 \]
for all $n\in\BZ$; they were originally proved in the first preprint 
version of \cite{AI:mm}.

Our second main result is a global version of part of
Theorem \ref{ithm:main}.  For a map of schemes $f\colon X\to Y$,
$f_0^{-1}\mathcal O_Y$ is a sheaf of commutative rings on~$X\<$, whose
stalk at any point $x\in X$ is $\mathcal O_{Y,\>f(x)\>}$ (see 
Section~\ref{Reduction over schemes}). The derived category of (sheaves
of) $f_0^{-1}\mathcal O_Y\>$-modules is denoted by $\D(f_0^{-1}\mathcal
O_Y)$.  Corollary \ref{formulae} of Theorem~\ref{thm:reduction3} gives:

\begin{itheorem}
\label{ithm:schemes}
Let\/ $f\col X\to Y$ be an essentially finite-type, flat map of noetherian
schemes; let\/ $X\xla{\pi_1}X\times_{Y}X\xra{\pi_2}X$ be the canonical
projections; let\/ $\delta\col X\to X\times_{Y}X$ be the diagonal
morphism; and let\/ $M$ and $N$ be complexes of\/  $\OX$-modules.

If\/ $M$ has coherent cohomology and is isomorphic in $\D(f_0^{-1}\mathcal
O_Y)$ to a bounded complex of $f_0^{-1}\mathcal O_Y$-modules that are
flat over~$Y\mkern-.5mu,$ and if $N$ has bounded-above quasi-coherent
homology, then one has an isomorphism
 \[
\delta^!(\pi_1^*M\Otimes{X\times_{Y}X}\pi_2^*N)\iso
\RH_X(\RH_{X}(M,f^{!}\mathcal O_{Y}),N)\,.
 \]
\end{itheorem}

When both schemes $X$ and $Y$ are affine, and $f$ corresponds to an
essentially finite-type ring homomorphism,  Theorem~\ref{ithm:schemes}
reduces to a special case of Theorem~\ref{ithm:main}, namely, where
the $K$-algebra $S$ is flat and $N$ is homologically bounded above.
In Section~\ref{Reduction over schemes} we also obtain global analogs
of other results proved earlier in the paper for complexes over rings.
A pattern emerging from these series of parallel results is that neither
version of a theorem implies the other one in full generality.  This
intriguing discrepancy suggests the existence of stronger global results.

The proofs of Theorems~~\ref{ithm:main} and~\ref{ithm:schemes} follow
very different routes.  The first one is based on isomorphisms in derived
categories of differential graded algebras; background material on the
topic is collected in Section~\ref{DG derived categories}.  The second
one involves fundamental results of Grothendieck duality theory,
systematically developed in \cite{H, Co2, Lp2}; the relevant notions
and theorems are reviewed in Section \ref{A:dual}.

\section{Relative dualizing complexes}
\label{Relative dualizing complexes}

\emph{In this section $\sigma\col K\to S$ denotes a homomorphism of 
commutative rings.}

For any $K$-algebra $P$ and each $n\in\BZ$ we write $\Omega_{P|K}$
for the $P$-module of K\"ahler differentials of $P$ over $K$, and set
$\Omega^n_{P|K}={\ts\bwedge}^{n}_P\Omega_{P|K}$ for each $n\in\BN$.

Recall that $\sigma$ is said to be \emph{essentially of finite type}
if it can be factored as
  \begin{equation}
 \label{eq:eft}
K\hra K[x_1,\dots,x_e]\to V^{-1}K[x_1,\dots,x_e]=Q\tra S\,,
   \end{equation}
where $x_1,\dots,x_e$ are indeterminates, $V$ is a multiplicatively
closed set, the first two maps are canonical, the equality defines $Q$,
and the last arrow is a surjective ring homomorphism.  We fix such a
factorization and set
 \begin{equation}
   \label{eq:reldual}
D^{\sigma}=\Shift^e\Rhom PS{\Omega^e_{Q|K}} \quad\text{in}\quad \dcat\,,
 \end{equation}
where $\dcat$ denotes the derived category of $S$-modules.  Any complex
isomorphic to $D^\sigma$ in $\dcat$ is called a \emph{relative dualizing
complex} of~$\sigma$.  To obtain such complexes we factor $\sigma$
through essentially smooth maps, see~\ref{ch:smooth}.

\begin{theorem}
  \label{thm:uniqueness}
If  $K\to P\to S$ is a factorization of $\sigma$, with $K\to P$
essentially smooth of relative dimension $d$ and $P\to S$ finite, then
there exists an isomorphism
  \[
D^{\sigma}\simeq \Shift^{d}\Rhom PS{\Omega^{d}_{P|K}}
      \quad\text{in}\quad \dcat\,.
  \]
\end{theorem}

The isomorphism in the theorem can be chosen in a coherent way for all
$K$-algebras essentially of finite type.  To prove this statement, or even
to  make it precise, we need to appeal to the theory of the pseudofunctor
${}^!$ of Grothendieck duality theory; see \cite[Ch.\,4]{Lp2}.
Canonicity is not used in this paper.

We write $\dcatf{\sigma}$ for the full subcategory of $\dcat$ consisting
of complexes $M\in\dcat$ such that  $\hh M$ is finite over~$S$ and $M$
is isomorphic in $\dcat[K]$ to some bounded complex of flat $K$-modules.

The name given to the complex $D^\sigma$ is explained by the next result.

\begin{theorem}
\label{thm:reduction1}
When $\fd_KS$ is finite the complex $D^\sigma$ has the following
properties.
 \begin{enumerate}[\rm(1)]
 \item
For each $M$ in $\dcatf{\sigma}$ the complex $\Rhom SM{D^{\sigma}}$ is in
$\dcatf{\sigma}$, and the biduality morphism gives a canonical
isomorphism
 \begin{equation*}
\delta^M\col M\simeq\Rhom S{\Rhom SM{D^{\sigma}}}{D^{\sigma}}
      \quad\text{in}\quad \dcat\,.
     \end{equation*}
 \item
One has $D^\sigma\in\dcatf{\sigma}$, and the homothety
map gives a canonical isomorphism
 \begin{equation*}
\chi^{D^\sigma}\col S\simeq\Rhom S{D^{\sigma}}{D^{\sigma}}
      \quad\text{in}\quad \dcat\,.
   \end{equation*}
 \end{enumerate}
\end{theorem}

The theorems are proved at the end of the section.  The arguments 
use various properties of (essentially) smooth homomorphisms, 
which we record next.

  \begin{chunk}
 \label{ch:smooth}
Let $\varkappa\col K\to P$ be a homomorphism of commutative 
noetherian rings.

One says that $\varkappa\col K\to P$ is (\emph{essentially}) \emph{smooth}
if it is (essentially) of finite type, flat, and the ring $k\otimes_KP$
is regular for each homomorphism of rings $K\to k$ when $k$ is a field;
see \cite[17.5.1]{Gr} for a proof that this notion of smoothness is
equivalent to that defined in terms of lifting of homomorphisms.

When $\varkappa$ is essentially smooth $\Omega_{P|K}^1$ is finite
projective, so for each prime ideal~$\fp$ of $P$ the $P_\fp$-module
$(\Omega_{P|K}^1)_\fp$ is free of finite rank.  If this rank is equal
to a fixed integer $d$ for all $\fp$, then $K\to P$ is said to be of
\emph{relative dimension} $d$; (essentially) smooth homomorphism of
relative dimension zero are called (\emph{essentially}) \emph{\'etale}.

   \begin{subchunk}
 \label{smooth:lambda}
Set $\env P=P\otimes_KP$ and $I=\Ker(\mu\col\env P\to P)$,
where $\mu$ is the multiplication.

There exist canonical isomorphisms of $P$-modules
  \[
\Omega_{P|K}^1\cong I/I^2\cong \Tor1{\env P}PP\,.
 \]
As $\mu$ is a homomorphism of commutative rings, $\Tor{}{\env P}PP$
has a natural structure of a strictly graded-commutative $P$-algebra,
so the composed isomorphism above extends to a homomorphism of graded
$P$-algebras
  \[
\lambda^{P|K}\col{\ts\bwedge}_P\Omega_{P|K}^1\lra\Tor{}{\env P}PP\,.
  \]
   \end{subchunk}

\begin{subchunk}
\label{smooth:tau}
Let $X\xra{\simeq}P$ be a projective resolution over $\env P$.
The morphism of complexes 
\begin{align*}
\delta\col X\otimes_{\env P}P&\to \Hom{\env P}{\Hom{\env P}X{\env P}}P
\\
\delta(x\otimes p)(\chi)&=(-1)^{(|x|+|p|)|\chi|}\chi(x)p
\end{align*}
yields the first map in the composition below, where $\kappa$ is a K\"unneth 
homomorphism:
\[
\xymatrixrowsep{.2pc} \xymatrixcolsep{2pc} \xymatrix{ 
\hh{X\otimes_{\env P}P} \ar@{->}[r]^-{\HH{}{\delta}}
 &\hh{\Hom{\env P}{\Hom{\env P}X{\env P}}P} 
\\
 {\phantom{\HH{}{X\otimes_{\env P}P}}} \ar@{->}[r]^-{\kappa} 
 & \Hom{\env P}{\HH{}{\Hom{\env P}X{\env P}}}P 
\\
 {\phantom{\HH{}{X\otimes_{\env P}P}}} \ar@{=}[r] 
 & \Hom{P}{\HH{}{\Hom{\env P}X{\env P}}}P\,.
 }
\]
Thus, one gets a homomorphism of graded $P$-modules
  \[
\tau^{P|K}\col \Tor{}{\env P}PP\lra \Hom P{\Ext{}{\env P}P{\env P}}P\,.
  \]
   \end{subchunk}

\begin{subchunk}
\label{smooth:eps}
The composition below, where the first arrow is a biduality map,
\[
\xymatrixrowsep{.2pc} \xymatrixcolsep{3pc} \xymatrix{ 
\Ext{}{\env P}P{\env P} \ar@{->}[r]
 & \Hom P{\Hom P{\Ext{}{\env P}P{\env P}}P}P
\\
 {\phantom{\Ext{}{\env P}P{\env P}}}
 \ar@{->}[r]^-{\Hom P{\tau^{P|K}}P} & \Hom P{\Tor{}{\env P}PP}P\,.
 }
\]
is a homomorphism of graded $P$-modules
   \[
\epsilon_{P|K}\col\Ext{}{\env P}P{\env P}\lra\Hom P{\Tor{}{\env P}P{\env P}}P\,.
  \]
\end{subchunk}

The maps above appear in homological characterizations of smoothness:

\begin{subchunk}
\label{smooth:criteria}
Let $K\to P$ be a flat and essentially of finite type homomorphism of
rings, and set $I=\Ker(\mu\col\env P\to P)$.  The following 
conditions are equivalent.
 \begin{enumerate}[\quad\rm(i)]
\item 
The homomorphism $K\to P$ is essentially smooth.
\item 
The ideal $I_\fm$ is generated by a regular sequence for each
prime ideal $\fm\supseteq I$.
\item 
The $P$-module $\Omega_{P|K}^1$ is projective and the map 
$ \lambda^{P|K}$ from \ref{smooth:lambda} is bijective.
\item 
The projective dimension $\pd_{\env P}P$ is finite.
\end{enumerate}

The equivalence of the first three conditions is due to Hochschild,
Kostant, and Rosenberg when $K$ is a perfect field, and to
Andr\'e~\cite[Prop.~C]{An} in general.  The implication (ii)$\implies$(iv)
is clear, and the converse is proved by Rodicio \cite[Cor.~2]{Ro}.
\end{subchunk}
  \end{chunk}

In the next lemma we use homological dimensions for complexes, as
introduced in \cite{AF:hd}. They are based on notions of semiprojective
and semiflat resolutions, recalled in \ref{semi:existence}. The
\emph{projective dimension} of $M\in\dcat[P]$ in defined by the formula
\[
\pd_PM=\inf\left\{n\in\BZ \left|\,
   \begin{gathered}
     \text{$n\ge \sup\hh M$ and $F\simeq M$ in $\dcat[P]$ with $F$}\\
     \text{semiprojective and $\coker(\dd^F_{n+1})$ projective}
   \end{gathered} \right\}\right.\,.
\]
The number obtained by replacing `semiprojective' with `semiflat' and 
`projective' with `flat' is the \emph{flat dimension} of $M$, denoted 
$\fd_PM$. 

\medskip

\emph{For the rest of this section we fix a factorization $K\to P\to S$
of $\sigma$, with $K\to P$ essentially smooth of relative dimension $d$
and $P\to S$ finite.}

\begin{lemma}
\label{smooth:fpd}
For every complex $M$ of $P$-modules the following inequalities hold:
\[
\fd_KM\le\fd_PM\le\fd_KM+\pd_{\env P}P\,.
\]
In particular, $\fd_PM$ and $\fd_KM$ are finite simultaneously.

When the $S$-module $\hh M$ is finite one can replace $\fd_PM$ with $\pd_PM$.
\end{lemma}

\begin{proof}
The inequality on the left is a consequence of \cite[4.2(F)]{AF:hd}.   

For the one on the right we may assume $\fd_KM=q<\infty$.  Thus, if
$F\to M$ is a semiflat resolution over $P$, then $G=\coker(\dd^F_{q+1})$
is flat as a $K$-module.  For each $n\in\BZ$ there is a canonical
isomorphism of functors of $P$-modules
\[
\Tor nPG - \cong \Tor n{\env P}P{G\otimes_K-}\,,
\]
see \cite[X.2.8]{CE}, so the desired inequality holds.  Since
$K\to P$ is essentially smooth one has $\pd_{\env P}P<\infty$,
see~\ref{smooth:criteria}, so they imply that $\fd_PM$ is finite if  only
if so is $\fd_KM$.  In case $\hh M$ is finite over $P$ one has $\fd_PM=
\pd_PM$; see~\cite[2.10(F)]{AF:hd}.
 \end{proof}

\begin{lemma}
\label{diagonal:smooth}
The canonical homomorphisms $\lambda^{P|K}_{d}$, $ \tau^{P|K}_d$, and 
$\epsilon_{P|K}^d$ defined in \emph{\ref{smooth:lambda}},
\emph{\ref{smooth:tau}}, and \emph{\ref{smooth:eps}}, respectively,
provide isomorphisms of $P$-modules
\begin{align}
\label{eq:smooth1}
\Ext{n}{\env P}P{\env P}&=0\quad\text{for}\quad n\ne d\,;
\\
\label{eq:smooth2}
{\Hom S{\lambda^{P|K}_{d}}P}\circ 
\epsilon_{P|K}^d \col \Ext{d}{\env P}P{\env P}
&\cong\Hom P{\Omega_{P|K}^d}P\,;
\\
\label{eq:smooth3}
\tau^{P|K}_d\circ\lambda^{P|K}_d \col \Omega_{P|K}^d
&\cong \Hom  P{\Ext{d}{\env P}P{\env P}}P\,.
\end{align}
\end{lemma}

\begin{proof}
Set $I=\Ker(\mu)$.  It suffices to prove that the maps above induce
isomorphisms after localization at every $\fn\in\Spec P$.  Fix one, then
set $T=P_\fm$, $R=\envv P{\fn\cap\env P}$ and $J=I_{\fn\cap\env P}$.  The
ideal $J$ is generated by a regular sequence, see~\ref{smooth:criteria}.
Any such sequence consists of $d$ elements:  This follows from the
isomorphisms of $T$-modules
 \[
J/J^2\cong(I/I^2)_\fn\cong(\Omega_{P|K}^1)_\fn\cong T^d\,.
 \]

The Koszul complex $Y$ on such a sequence is a free resolution of $T$ over
$R$.  A well known isomorphism $\Hom RYR\cong\Shift^{-d}Y$ of complexes
of $R$-modules yields $\Ext nRTR=0$ for $n\ne d$ and $\Ext dRTR\cong T$.
This establishes \eqref{eq:smooth1} and shows that $\Ext{d}{\env P}P{\env
P}$ is invertible; as a consequence, \eqref{eq:smooth2} follows from
\eqref{eq:smooth3}.

We analyze the maps in \eqref{eq:smooth3}.  From~\ref{smooth:criteria}
we know that $\lambda^{P|K}_d$ is bijective.  By~\ref{smooth:tau} one
has $\tau^{P|K}_d=\kappa_d\circ\HH d{\delta}$.  The map $\HH d{\delta}$
is bijective, as it can be computed from a resolution $X$ of $P$ by finite
projective $\env P$-modules, and then  $\delta$ itself is an isomorphism.
To establish the isomorphism in \eqref{eq:smooth3} it remains to show that
$(\kappa_d)_{\fm}$ is bijective.  This is a {K\"unneth} map, which can
be computed using the Koszul complex $Y$ above.   Thus, we need to show
that the natural $T$-linear map
\[ 
\HH d{\Hom{R}{\Hom RYR}T}\lra\Hom R{\HH{-d}{\Hom RYR}}T
\]
is bijective.  It has been noted above that both modules involved are isomorphic 
to~$T$, and an easy calculation shows that the map itself is an isomorphism.
  \end{proof}

To continue we need a lemma from general homological algebra.

  \begin{lemma}
 \label{homology:projective}
Let $R$ be an associative ring and $M$ a complex of $R$-modules.

If the graded $R$-module $\hh M$ is projective, then 
there exists a unique up to homotopy morphism of complexes $\hh M\to M$ 
inducing $\id^{\hh M}$, and a unique isomorphism $\alpha\col \hh M\to M$ 
in $\dcat[R]$ with $\hh\alpha=\id^{\hh M}$.
  \end{lemma}

\begin{proof}
One has $\hh M\cong\coprod_{i\in\BZ}\Shift^i\HH iM$ as complexes  
with zero differentials.  The projectivity of the $R$-modules $\HH iM$ 
provides the second link in the chain 
\begin{align*}
\hh{\Hom R{\hh M}M}
&\cong\hh{\prod_{i\in\BZ}\Shift^{-i}\Hom R{\HH iM}M}\\
&\cong\prod_{i\in\BZ}\Shift^{-i}\Hom R{\HH iM}{\hh M}\\
&\cong\Hom R{\coprod_{i\in\BZ}\Shift^i\HH iM}{\hh M}\\
&\cong\Hom R{\hh M}{\hh M}
\end{align*}
of isomorphisms of graded modules.  The composite map is given by
$\cls\alpha\mapsto\hh\alpha$.  The first assertion follows because
$\HH 0{\Hom R{\hh M}M}$ is the set of homotopy classes of morphisms
$\hh M\to M$. For the second, note that one has \[ \Mor{\dcat[R]}{\hh
M}M\cong\HH 0{\Hom R{\hh M}M} \] because  each complex $\Shift^i\HH iM$
is semiprojective, and hence so is $\hh M$.
  \end{proof}

\begin{lemma}
\label{cor:smooth}
In $\dcat[P]$ there exist canonical isomorphisms
\begin{align}
\Rhom{\env P}P{\env P}&\simeq\Shift^{-d}\Hom P{\Omega_{P|K}^d}P\,.
\label{eq:smooth4}
\\
\label{eq:smooth5}
\Rhom P{\Rhom{\env P}P{\env P}}P&\simeq\Shift^{d}\Omega_{P|K}^d\,.
\end{align}
\end{lemma}

\begin{proof}
Since $K\to P$ is essentially smooth of relative dimension $d$, the
$P$-module $\Omega_{P|K}^d$ is projective of rank one, and hence so
is $\Hom P{\Omega_{P|K}^d}P$.  The isomorphisms \eqref{eq:smooth1}
and \eqref{eq:smooth2} imply that $\hh{\Rhom{\env P}P{\env P}}$ is
an invertible graded $P$-module.  In particular, it is projective.
Now choose \eqref{eq:smooth4} to be the canonical isomorphism provided by
Lemma \ref{homology:projective}, and  \eqref{eq:smooth5} the isomorphism
induced by it.  \end{proof}

\begin{lemma}
 \label{lem:finite_factorization}
When $\sigma$ is finite there is a canonical isomorphism
\[
\Shift^{d}\Rhom PS{\Omega^{d}_{P|K}}\cong \Rhom KSK
  \quad\text{in}\quad \dcat\,.
\]
  \end{lemma}

\begin{proof} 
One has a chain of canonical isomorphisms:
\begin{align*}
\Shift^{d}\Rhom PS{\Omega^{d}_{P|K}}
   &\simeq \Shift^{d}\Rhom{\env P}P{\Rhom KS{\Omega^{d}_{P|K}}}\\
   &\simeq \Shift^{d}\Rhom{\env P}P{\env P}\dtensor{\env P}{\Rhom KS{\Omega^{d}_{P|K}}}\\
   &\simeq \Rhom P{\Omega^{d}_{P|K}}P\dtensor{\env P}{\Rhom KS{\Omega^{d}_{P|K}}}\\
   &\simeq \Rhom P{\Omega^{d}_{P|K}}P
     \dtensor{\env P}\big({\Omega^{d}_{P|K}}\dtensor K {\Rhom KSK}\big)\\
   &\simeq \Rhom P{\Omega^{d}_{P|K}}P\dtensor P
        \big(P\dtensor{\env P}\big({{\Omega^{d}_{P|K}}}\dtensor K {\Rhom KSK}\big)\big)\\
   &\simeq \Rhom P{\Omega^{d}_{P|K}}P\dtensor P  
        \big({{\Omega^{d}_{P|K}}}\dtensor P {\Rhom KSK}\big) \\
   &\simeq \Rhom P{\Omega^{d}_{P|K}} {{\Omega^{d}_{P|K}}} \dtensor P {\Rhom KSK}\\
   &\simeq \Rhom KSK\,.         
\end{align*}
The first one holds by a classical associativity formula,
see \eqref{eq:isos1}, the second one because $\pd_{\env
P}P$ is finite, see \ref{smooth:criteria}, the third one
by \eqref{eq:smooth4}.  The last one is induced by the homothety
$P\to\Rhom P{\Omega^{d}_{P|K}}{{\Omega^{d}_{P|K}}}$, which is bijective
as $(\Omega^{d}_{P|K})_\fp\cong P_\fp$ holds as $P_\fp$-modules for each
$\fp\in\Spec P$.  The other isomorphisms are standard.
  \end{proof}

\stepcounter{theorem}

\begin{proof}[Proof of Theorem \emph{\ref{thm:uniqueness}}] 
Let $K\to Q\to S$ be the factorization of $\sigma$ given by
\eqref{eq:eft}, with $Q=V^{-1}K[x_1,\dots,x_e]$.  The isomorphism
  \[
\Omega^{1}_{(P\otimes_KQ)|K}\cong
(\Omega^{1}_{P|K}\otimes_KQ)\oplus(P\otimes_K\Omega^{1}_{Q|K})
  \]
induces the first isomorphism of $(P\otimes_KQ)$-modules below:
 \begin{align*}
\Omega^{d+e}_{(P\otimes_KQ)|K}&\cong
\bigoplus_{i+j=d+e}(\Omega^{i}_{P|K}\otimes_KQ)\otimes_{P\otimes_KQ}
(P\otimes_K\Omega^{j}_{Q|K})\\ &\cong
\Omega^{d}_{P|K}\otimes_{K}\Omega^{e}_{Q|K}\,.
  \end{align*}
The second one holds because for each $\fp\in\Spec P$ one has
$(\Omega^{i}_{P|K})_\fp\cong\wedge^i_{P\fp}(P_\fp^d)=0$ for $i>d$,
and similarly $(\Omega^{i}_{Q|K})_\fp=0$ for $j>e$.  One also has
  \begin{equation}
    \label{eq:basechange}
\Omega^{n}_{(P\otimes_KQ)|Q}\cong \Omega^{n}_{P|K}\otimes_{K}Q
  \quad\text{for every}\quad n\in\BN\,.
  \end{equation}

The isomorphisms above explain the first and third links in the chain
\begin{align*}
\Rhom{P\otimes_{K}Q}S{\Shift^{d+e}\Omega^{d+e}_{(P\otimes_{K}Q)|K}}
 &\simeq\Rhom{P\otimes_{K}Q}S{\Shift^{d}\Omega^{d}_{P|K}\otimes_{K}\Shift^{e}\Omega^{e}_{Q|K}} \\
 &\simeq \Rhom{P\otimes_{K}Q}S{\Shift^{d}\Omega^{d}_{P|K}\otimes_{K} Q} \otimes_{Q}\Shift^{e}\Omega^{e}_{Q|K}\\
 &\simeq \Rhom{P\otimes_{K}Q}S{\Shift^{d}\Omega^{d}_{(P\otimes_KQ)|Q}} \otimes_{Q}\Shift^{e}\Omega^{e}_{Q|K}\\
 &\simeq \Rhom QSQ\otimes_{Q}\Shift^{e}\Omega^{e}_{Q|K} \\
 &\simeq \Rhom QS{\Shift^{e}\Omega^{e}_{Q|K}}
\end{align*}
For the fourth isomorphism, apply Lemma \ref{lem:finite_factorization} to 
the factorization $Q\to P\otimes_{K}Q\to S$ of the finite homomorphism $Q\to S$,
where the first map is essentially smooth by \cite[17.7.4(v)]{Gr} and has
relative dimension $d$ by \eqref{eq:basechange}.  The other isomorphisms are standard.  
By symmetry one also obtains an isomorphism 
  \[
\Rhom{P\otimes_{K}Q}S{\Shift^{d+e}\Omega^{d+e}_{(P\otimes_{K}Q)|K}}
\simeq \Rhom PS{\Shift^{d}\Omega^{d}_{P|K}}\,.
  \qedhere
  \]
  \end{proof}

\begin{proof}[Proof of Theorem \emph{\ref{thm:reduction1}}]
Recall that $K\to P\to S$ is a factorization of $\sigma$ with $K\to P$
essentially smooth of relative dimension $d$ and $P\to S$ finite. Set
$L=\Shift^{d}\Omega^{d}_{P|K}$, and note that one has $D^{\sigma}=\Rhom
PSL$; see Theorem~\ref{thm:uniqueness}.

(1)  Standard adjunctions give isomorphisms of functors
\[
\Rhom S-{D^{\sigma}}\cong\Rhom S-{\Rhom PS{L}}\cong\Rhom P-{L}\,,
\]
For $M\in\dcatf{\sigma}$ Lemma~\ref{smooth:fpd} yields $\pd_PM<\infty$,
so $M$ is represented in $\dcat[P]$ by a bounded complex $F$ of finite
projective $P$-modules.  As $L$ is a shift of a finite  projective
$P$-module, $\Hom PF{L}$ is a bounded complex of finite projective
$P$-modules.  It represents $\Rhom PM{L}$, so one sees that $\hh{\Rhom
PM{L}}$ is finite over $P$.  As $P$ acts on  it through $S$, it is
finite over $S$ as well; furthermore, $\fd_K{\Rhom PM{L}}$ is finite
by Lemma~\ref{smooth:fpd}.

The map $\delta^M$ in $\dcat$ is represented in $\dcat[P]$ by the
canonical biduality map
  \[
F\to\Hom P{\Hom PF{L}}{L}\,.
  \]
This is a quasiisomorphism as $F$ is finite complex of finite projectives
and $L$ is invertible. It follows that $\delta^M$ is an isomorphism.

(2)  Since $\fd_KS$ is finite, (1) applied to $M=S$ shows that
$D^{\sigma}=\Rhom SS{D^{\sigma}}$ is in $\dcatf{\sigma}$ and that
$\delta^S\col S\to \Rhom S{\Rhom SS{D^\sigma}}{D^\sigma}$ is an
isomorphism.  Composing $\delta^S$ with the map induced by the isomorphism
$D^\sigma\simeq\Rhom SS{D^\sigma}$ one gets $\chi^{D^\sigma}\col S\to
\Rhom S{D^\sigma}{D^\sigma}$, hence $\chi^{D^\sigma}$ is an isomorphism.
 \end{proof}

\section{DG derived categories}
\label{DG derived categories}

Our purpose here is to introduce background material on differential
graded homological algebra needed to state and prove the results
in Sections~\ref{Derived Hochschild functors over algebras} and
\ref{Reduction over algebras}.

\emph{In this section $K$ denotes a commutative ring.}

\begin{chunk}\textbf{DG algebras and DG modules.}\ 
Our terminology and conventions generally agree with those of MacLane
\cite[Ch.\,VI]{Ml}.  All DG algebras are defined over $K$, are zero in
negative degrees, and act on their DG modules from the left.  When $A$
is a DG algebra and $N$ a DG $A$-module we write $A\nat$ and $N\nat$
for the graded algebra and graded $ A\nat$-module underlying $A$ and
$N$, respectively.  We set
  \begin{align*}
\inf N&=\inf\{n\in\BZ\mid N_n\ne0\}\,;
 \\
\sup N&=\sup\{n\in\BZ\mid N_n\ne0\}\,.
  \end{align*}
Every element $x\in N$ has a well defined degree, denoted $|x|$.

When $B$ is a DG algebra the complex $A\otimes_KB$ is a DG algebra with
product $(a\otimes b)\cdot (a'\otimes b')=(-1)^{|b||a'| }(aa'\otimes
bb')$.

When $M'$ is a DG $B$-module the complex $N\otimes_KM'$ is
canonically a DG module over $A\otimes_KB$, with $(a\otimes
b)\cdot(n\otimes m')=(-1)^{|b||n|}an\otimes bm'$.

The \emph{opposite DG $K$-algebra} $\op A$ has the same underlying
complex of $K$-modules as $A$, and product $\cdot$ given by $a\cdot
b= (-1)^{|a||b|}ba$.  We identify right DG $A$-modules with DG 
modules over $\op A$, via the formula $am=(-1)^{|a||m|}ma$.  

When $M$ is a DG $B$-module the complex $\Hom KMN$ is canonically
a DG $A\otimes_K\op B$-module, with action given by $\big((a\otimes
b)(\alpha)\big)(m)=(-1)^{|b||\alpha|}a\alpha(bm)$.

We write $\env A$ for the DG $K$-algebra $A\otimes_{K}\op A$. Any morphism
$\alpha\col A\to B$ of DG $K$-algebras induces a morphism $\env\alpha =
\alpha\otimes_{K}\op\alpha$ from $\env A$ to $\env B$.  There is a natural
DG $\env A$-module structure on $A$ given by $(a\otimes a')x=(-1)^{|a'||x|}axa'$.

For every DG $A\otimes_K\op B$-module $L$, \cite[VI.(8.7)]{Ml} yields a 
canonical isomorphism
  \begin{equation}
  \label{eq:isos1}
\Hom {A\otimes_K\op B}L{\Hom KMN} \cong  \Hom A{L\otimes_BM}N\,.
  \end{equation}
For every DG $\op A\otimes_KB$-module $L'$, \cite[VI.(8.3)]{Ml} yields a 
canonical isomorphism
  \begin{equation}
  \label{eq:isos2}
L'\otimes_{A\otimes_K\op B}(N\otimes_KM')\cong(L'\otimes_A N)\otimes_BM'\,.
  \end{equation}
  \end{chunk}
  
\begin{chunk}\textbf{Properties of DG modules.}\ 
A DG $A$-module $F$ is said to be \emph{semiprojective} if the
functor $\Hom AF-$ preserves surjections and quasi-isomorphisms,
and \emph{semiflat} if $(F\otimes_A-)$ preserves injections and
quasi-isomorphisms.  If $F$ is semiprojective, respectively, semiflat, 
then $F\nat$ is projective, respectively, flat, over $A\nat$; the converse 
is true when $F$ is bounded below.  Semiprojectivity implies semiflatness.  

A DG module $I$ is \emph{semiinjective} if $\Hom A-I$ transforms
injections into surjections and preserves quasi-isomorphisms.  If $I$
is semiinjective, then $I\nat$ is injective over $A\nat$; the converse
is true when $I$ is bounded above.

\begin{subchunk}
\label{semi:equivalences}
Every quasi-isomorphisms of DG modules, both of which are either 
semiprojective or semiinjective, is a homotopy equivalence.
  \end{subchunk}

The following properties readily follow from standard adjunction formulas.

\begin{subchunk}
\label{semi:basechange}
Let $\alpha\col A\to B$ be a morphism of DG $K$-algebras, and let $X$ and $Y$ be DG
modules over $A$ and $B$, respectively. The following statements hold:
\begin{enumerate}[\rm(1)]
\item If $X$ is semiprojective, then so is the DG $B$-module $B\otimes_AX$.
\item If $X$ is semiinjective, then so is the DG $B$-module $\Hom ABX$.
\item If $B$ is semiprojective over $A$ and $Y$ is semiprojective over $B$, then $Y$ is
 semiprojective over $A$.
\item If $B$ is semiflat over $A$ and $Y$ is semiinjective over $B$, then $Y$ is
 semiinjective over $A$.
\end{enumerate}
  \end{subchunk}
   \end{chunk}
   
\begin{chunk}\textbf{Resolutions of DG modules.}\  
Let $M$ be a DG $A$-module.

 \begin{subchunk}
\label{semi:existence}
A semiprojective resolution of $M$ is a quasi-isomorphism $F\tiso M$ with
$F$ semiprojective.  Each DG $A$-module $M$ admits such a resolution;
\cite[\S 1]{AH}.

A semiinjective resolution of $M$ is a quasi-isomorphism $M\tiso I$ with
$I$ semiinjective.  Every DG $A$-module $M$ admits such a resolution;
see \cite[\S3-2]{Ke}.

In what follows, for each DG module $M$ over $A$, we fix a semiprojective 
resolution $\pi_{A}^{M}\col \projres_{A}(M)\to M$, 
and a semiinjective resolution $\iota_{A}^{M}\col M\to \injres_{A}(M)$. 

Each morphism of DG modules lifts up to homotopy to a morphism of
their semiprojective resolutions and extends to a morphism of their
semiinjective resolutions, and such a lifting or extension is unique up
to homotopy.  In particular, both $F$ and $I$ are unique up to homotopy
equivalences inducing the identity on $M$.  \end{subchunk}
   \end{chunk}

 \begin{mysublemma}
   \label{lem:change}
Let $\omega\col A\to B$ be a quasi-isomorphism of DG algebras, $I$ a
semiinjective DG $A$-module, $J$ a semiinjective DG $B$-module, 
and $\iota\col J\to I$ a quasi-isomorphism of DG $A$-modules.

For every DG $B$-module $L$ the following map is a 
quasi-isomorphism: 
  \[
\Hom{\omega}L{\iota}\col\Hom BLJ\to\Hom ALI\,.
  \]
 \end{mysublemma}

\begin{proof}
The morphism $\iota$ factors as a composition
  \[
J\xra{\,\iota'\,}\Hom ABI
\xra{\,\Hom A{\omega}I\,}\Hom AAI\cong I
  \]
of morphisms of DG $A$-modules, where $\iota'(x)(b)=(-1)^{|x||b|}b\iota(x)$.  
It follows that $\iota'$ is a quasi-isomorphism.  Now $J$ is a semiinjective
DG $B$-module by hypothesis, $\Hom ABJ$ is one by \ref{semi:basechange}(2),
so \ref{semi:equivalences} yields
  \[
\Hom BLJ\xra{\,\simeq\,}
\Hom BL{\Hom ABI}\cong
\Hom ALI\,.
  \]
It remains to note that the composition of these maps is equal to
$\Hom{\omega}L{\iota}$.
  \end{proof}

 \begin{mysublemma}
   \label{quisms}
Let $\omega\col A\to B$ be a morphism of DG algebras, and let $Y$ and $Y'$ be DG
$B$-modules that are quasi-isomorphic when viewed as DG $A$-modules.

If $\omega$ is a quasi-isomorphism, or if there exists a morphism $\beta\col B\to A$,
such that $\omega\beta=\id^B$, then $Y$ and $Y'$ are quasi-isomorphic as DG $B$-modules.
 \end{mysublemma}

\begin{proof}
By hypothesis, one has $A$-linear quasi-isomorphisms
$Y\xla{\upsilon}U\xra{\upsilon'}Y'$.

When $\omega$ is a quasi-isomorphism, choose $U$ semiprojective over $
A$, using~\ref{semi:existence}.  With vertical arrows defined to be $b\otimes 
u\mapsto b\upsilon(u)$ and $b\otimes u\mapsto b\upsilon'(u)$ the diagram
\[
\xymatrixrowsep{1pc} \xymatrixcolsep{1.2pc}
\xymatrix{
&&&Y
  \\
U
\ar@/^0.8pc/[urrr]_(.70){\simeq}^(.25){\upsilon}
\ar@/_0.8pc/[drrr]^(.70){\simeq}_(.25){\upsilon'}
\ar@{=}[r]
&A\otimes_AU
\ar@{->}[rr]^{\omega\otimes_AU}_{\simeq}
&&B\otimes_AU
\ar@{->}[u]
\ar@{->}[d]
  \\
&&&Y'
}
\]
commutes.  The vertical maps are morphisms of DG $B$-modules, 
and $\omega\otimes_AU$ is a quasi-isomorphism because $\omega$ 
is one and $U$ is semiprojective.

When $\omega$ has a right inverse $\beta$, note that the $A$-linear 
quasi-isomorphisms $\upsilon$ and $\upsilon'$ are also $B$-linear, 
and that the DG $B$-module structures on $Y$ and $Y'$ induced 
via $\beta$ are identical with their original structures over $B$.
 \end{proof}

We recall basic facts concerning DG derived categories; see Keller
\cite{Ke} for details.

\begin{chunk}\textbf{DG derived categories.}
Let $A$ be a DG algebra and $M$ a DG $A$-module.  

DG $A$-modules and their morphisms form an abelian category.  The
\emph{derived category} $\dcat[A]$ is obtained by keeping the same objects
and by formally inverting all quasi-isomorphisms.  It has a natural
triangulation, with translation functor $\Shift$ is defined on $M$ by
$(\Shift M)_i=M_{i-1}$, $\dd^{\Shift M}\varsigma(m)=-\varsigma(\dd^M(m))$,
and $a\varsigma(m)=(-1)^{|a|}\varsigma(am)$, where $\varsigma\col
M\to\Shift M$ is the degree one map given by $\varsigma(m)=m$.

For any semiprojective resolution $F\to M$, and each $N\in\dcat[A]$
one has
  \[
\Mor{\dcat[R]}MN\cong\HH 0{\Hom RFN}\,.
  \]

\begin{subchunk}
\label{derived functors}
For all $L\in\dcat[\op A]$ and $M,N$ in $\dcat[A]$, the complexes of
$K$-modules
 \[
L\dtensor AM = L\otimes_AF \qquad\text{and}\qquad \Rhom AMN = \Hom AFN
 \]
are defined uniquely up to unique isomorphisms in $\dcat [A]$.
When $\omega\col A\to B$ is a morphism of DG algebras, $L'$, $M'$ and 
$N'$ are DG $B$-modules, and $\lambda\col L\to
L'$, $\mu\col M\to M'$, and $\nu\col N'\to N$ are $\omega$-equivariant
morphisms of DG modules, there exist uniquely defined morphisms
  \begin{align*}
{\lambda}\dtensor{\omega}{\mu}\col&L\dtensor AM\to L'\dtensor B{M'}\,,
  \\
\Rhom{\omega}{\mu}{\nu}\col&\Rhom B{M'}{N'}\to\Rhom AMN\,.
  \end{align*}
that depend functorially on all three arguments, and are isomorphisms
when all the morphisms involved have this property. For each $i\in\BZ$
one sets
  \[
\Tor iALM = \HH i{L\dtensor AM} \quad\text{and}\quad
\Ext iAMN =\HH {-i}{\Rhom AMN}\,.
  \]
    \end{subchunk}

\begin{subchunk}
  \label{DGoverRings}
Associative $K$-algebras are viewed as DG algebras concentrated 
in degree zero, in which case DG modules  are simply complexes 
of left modules.  Graded modules are complexes with zero differential,
and modules are complexes concentrated in degree zero.  The constructions
above specialize to familiar concepts:

When $A_i=0$ for $i\ne0$ the derived category $\dcat[A]$ 
coincides with the classical unbounded derived category of the category of 
$A_0$-modules.   Similarly, if $M$ and $N$ are DG $A$-modules with 
$M_i=0=N_i$ for $i\ne0$, then for all $n\in\BZ$ one has 
$\Ext nAMN=\Ext n{A_0}{M_0}{N_0}$ and $\Tor nAMN=\Tor n{A_0}{M_0}{N_0}$.  
  \end{subchunk}

\begin{subchunk}
\label{DGchangeRings}
Let $\omega\col A\to B$ be a morphism of DG algebras. Viewing DG $B$-modules as DG $A$-modules via restriction along $\omega$, one gets a functor of derived categories
\[
\omega^{*}\col \dcat[B]\to \dcat[A]\,.
\]
When $\omega$ is a quasi-isomorphism it is an equivalence,
with quasi-inverse $B\dtensor A-$.
\end{subchunk}
  \end{chunk}

\section{Derived Hochschild functors}
\label{Derived Hochschild functors over algebras}

In this section we explain the left hand side of the isomorphism in 
Theorem~\ref{ithm:main}.

Let $K$ be a commutative ring and $\sigma\col K\to S$ an associative
$K$-algebra.

  \begin{chunk}
    \label{hoch:res}
A \emph{flat DG algebra resolution} of $\sigma$ is a factorization
$K\to A\xra{\alpha} S$ of $\sigma$ as a composition of morphisms of
DG algebras, where each $K$-module $A_i$ is flat and $\alpha$ is a
quasi-isomorphism; complexes of $S$-modules are viewed as 
DG $A$-modules via $\alpha$.  When $K\to B\xra{\beta} S$ is a flat 
DG algebra resolution of $\sigma$, we say that $\omega\col A\to B$ 
is a \emph{morphism of resolutions} if it is a morphism of DG 
$K$-algebras, satisfying $\beta\omega=\alpha$.

We set $\env A=A\otimes_K\op A$, note that $K\to\op A\xra{\op\alpha}\op
S$ is a flat DG algebra resolution of $\op\sigma\col K\to\op
S$, and turn $S$ into a DG module over $\env A$ by $(a\otimes
a')s=\alpha(a)s\,\op\alpha(a')$.

Flat DG algebra resolutions always exist: A resolution $K\to T\to S$,
with $T\nat$ the tensor algebra of some free non-negatively graded
$K$-module, can be obtained by inductively adjoining noncommuting
variables to $K$; see also Lemma \ref{hoch:uniqueness}.
  \end{chunk}

Here we construct one of four functors of pairs of complexes of
$S$-modules that can be obtained by combining $\Rhom{\env A}S-$
and $S\dtensor{\env A}-$ with $(-\dtensor K-)$ and $\Rhom K--$.
The other three functors are briefly discussed in \ref{hoch:homology}
and \ref{DGAres:prereduction}.

The statement of the following theorem is related to results in
\cite[\S 2]{YZ1}.  We provide a detailed proof, for reasons explained
in \ref{rem:YZ}.

  \begin{theorem}
    \label{hoch:main}
Each flat DG algebra resolution $K\to A\to S$ of $\sigma$  
defines a functor 
\[
\Rhom{\env A}S{-\dtensor K-}\col\dcat\times\dcat[\op S]\to\dcat[S^{\sfc}]\,,
\]
where $S^{\sfc}$ denote the center of $S$, described by 
\eqref{eq:derivedHoch}.  For every flat DG
algebra resolution $K\to B\to S$ of $\sigma$ there is a canonical
natural equivalence of functors
  \[
\omega^{AB}\col\Rhom{\env A}S{-\dtensor K-}\to\Rhom{\env B}S{-\dtensor K-}\,,
  \]
given by \eqref{eq:nattransf2}, and every flat DG algebra resolution
$K\to C\to S$ of $\sigma$ satisfies
  \[
\omega^{AC}=\omega^{BC}\omega^{AB}\,.
  \] \end{theorem}

The theorem validates the following notation: 

  \begin{remark}
Fix a flat DG algebra resolution $K\to A\to S$ of $\sigma$ and let
   \[
\Rhom{S\dtensor K\op S}S{-\dtensor K-}\col\dcat\times\dcat[\op
S]\to\dcat[S^{\sfc}]
  \]
denote the functor $\Rhom{A\otimes_K{\op A}}S{-\dtensor K-}$.
For all $L\in\dcat$ and $L'\in\dcat[\op S]$ it yields \emph{derived 
Hochschild cohomology modules} with tensor-decomposable 
coefficients:
  \[
\Ext n{S\dtensor K\op S}S{L\dtensor K{L'}}
=\CH n{\Rhom{S\dtensor K\op S}S{L\dtensor K{L'}}}\,.
  \]
These modules are related to vintage Hochschild cohomology.

For all $S$-modules $L$ and $L'$ there are canonical natural maps
 \[
\hcoh nKS{L\otimes_KL'}\to\Ext n{S\otimes_K\op S}S{L\otimes_KL'}
  \]
of $S^{\sfc}$-modules, where the modules on the left are the classical ones, 
see \ref{DGoverRings}. These are isomorphisms when $S$ is $K$-projective; 
see \cite[IX, \S6]{CE}. When one of $L$ or $L'$ is $K$-flat, there exist 
canonical natural homomorphisms 
 \[
\Ext n{\alpha\otimes_K\op\alpha}S{L\otimes_KL'}\col
\Ext n{S\otimes_K\op S}S{L\otimes_KL'}
\to
\Ext n{S\dtensor K\op S}S{L\otimes_KL'}\,.
  \]
When $S$ is $K$-flat the composition $K\to S\xra{=}S$ is a flat DG resolution 
of $\sigma$ and $\alpha\col A\to S$ is a morphism of resolutions, so the 
theorem shows that the maps above are isomorphisms.
   \end{remark}

\begin{construction}
  \label{hoch:comparison} 
Let $K\to A\xra{\,\alpha\,} S$ and $K\to {A'}\xra{\alpha'}\op
S$ be flat DG algebra resolutions of $\sigma$ and of $\op\sigma$,
respectively.  We turn $S$ into a DG module over $A\otimes_KA'$
by setting $(a\otimes a')s=\alpha(a)s\,\alpha'(a')$.  The action of
$S^{\sfc}$ on $S$ commutes with that of $A\otimes_K{A'}$, and so 
confers a natural structure of complex of $S^{\sfc}$-modules on
  \[
\Hom{A\otimes_K{A'}}S{\injres_{A\otimes_K{A'}}(\projres_A(L)\otimes_K\projres_{A'}(L')}\,,
  \]
where $\projres_A$ and $\injres_{A\otimes_K{A'}}$ refer to the resolutions 
introduced in \ref{semi:existence}.

Let $K\to B\xra{\beta} S$ and $K\to {B'}\xra{\beta'}\op S$ be DG algebra
resolutions of $\sigma$ and $\op\sigma$, respectively, and $\omega\col
A\to B$ and $\omega'\col A'\to B'$ be morphism of resolutions.  We turn DG
$B$-modules into DG $A$-modules via $\omega$, and remark that the equality
$\beta\omega=\alpha$ implies that on $S$-modules the new action of $A$
coincides with the old one.

Let  $\lambda\col L\to M$
be a morphism of DG $S$-modules and $\lambda'\col L'\to M'$ one of DG
$\op S$-modules.  The lifting property of semiprojective DG modules
yields diagrams
  \begin{equation}
  \label{eq:twodiagrams}
\begin{gathered}
\xymatrixrowsep{2pc}
\xymatrixcolsep{2pc}
\xymatrix{
\projres_{A}(L)
\ar@{->}[d]_{\simeq}\ar@{->}[r]^{\wt\lambda}
&\projres_{B}(M)
\ar@{->}[d]^{\simeq}
  \\
L\ar@{->}[r]^{\lambda}
&M}
  \end{gathered}
  \quad\text{and}\quad
\begin{gathered}
\xymatrixrowsep{2pc}
\xymatrixcolsep{2pc}
\xymatrix{
\projres_{A'}(L')
\ar@{->}[d]_{\simeq}\ar@{->}[r]^{\wt\lambda'}
&\projres_{B'}(M')
\ar@{->}[d]^{\simeq}
  \\
L'\ar@{->}[r]^{\lambda'}
&M'}
  \end{gathered}
    \end{equation}
of DG $A$-modules and DG $A'$-modules, respectively, that commute
up to homotopy.  It provides the morphism in the top row of a diagram of
DG $(A\otimes_K{A'})$-modules
  \begin{equation}
  \label{eq:onediagram}
    \begin{gathered}
    \xymatrixrowsep{2pc}
\xymatrixcolsep{1.3pc}
\xymatrix{
\projres_{A}(L)\otimes_{K}\projres_{A'}(L')\ar@{->}[r]^{\wt\lambda\otimes_{K}\lambda'}
\ar@{->}[d]_{\simeq}\ar@{->}[r]^{\wt\lambda\otimes_{K}\wt\lambda'}
&\projres_{B}(M)\otimes_{K}\projres_{B'}(M')
\ar@{->}[dd]^{\simeq}
  \\
\injres_{A\otimes_K{A'}}(\projres_{A}(L)\otimes_{K}\projres_{A'}(L'))
 \\
\injres_{A\otimes_K{A'}}\big(\injres_{B\otimes_K{B'}}
(\projres_{B}(M)\otimes_{K}\projres_{B'}(M'))\big)
\ar@{<-}[u]_-{\epsilon}
&\injres_{B\otimes_K{B'}}(\projres_{B}(M)\otimes_{K}\projres_{B'}(M'))
\ar@{->}[l]_-{\iota}^-{\simeq}
}
     \end{gathered}
  \end{equation}
that commutes up  to homotopy, where $\iota$ is the chosen semiinjective
resolution, and $\epsilon$ is given by the extension property of
semiinjective DG module over $A\otimes_KA'$; for conciseness, we rewrite
these maps as $E\xra{\epsilon}I\xla{\iota}J$.  They are unique up to
homotopy, as the liftings and extensions used for their construction
have this property.

The hypotheses $\beta\omega=\alpha$ and $\beta'\omega'=\alpha'$
imply that $\omega$ and $\omega'$ are quasi-isomor\-phisms, hence so is
$\omega\otimes_K\omega'$, due to the $K$-flatness of $A\nat$ and $B'\nat$.
Since $\iota$ is a quasi-isomorphism, Lemma \ref{lem:change} shows that
so is $\Hom{\omega\otimes_K{\omega'}}S{\iota}$; thus, the latter map
defines in $\dcat[S^{\sfc}]$ an isomorphism,
denoted $\Rhom{\omega\otimes_K{\omega'}}S{\iota}$.  We set
  \begin{multline}
    \label{eq:comparison}
[\omega,\omega'](\lambda,\lambda')=
\Rhom{\omega\otimes_K{\omega'}}S{\iota}^{-1} \circ \Rhom{A\otimes_K{A'}}S{\epsilon}\col\\
\Rhom{A\otimes_K{A'}}S{L\dtensor K{L'}}\lra
\Rhom{B\otimes_K{B'}}S{M\dtensor K{M'}}
  \end{multline}
  \end{construction}

The first statement of the following lemma contains the existence of
the functors $\Rhom{\env A}S{-\dtensor K-}$, asserted in the theorem.
The second statement, concerning the uniqueness of these functors,
is weaker than the desired one, because it only applies to resolutions
that can be compared through a morphism $\omega\col A\to B$.  On the
other hand, it allows one to compare functors defined by independently
chosen resolutions of $\sigma$ and $\op\sigma$.  The extra generality
is needed in the proof of Lemma \ref{hoch:uniqueness}.

\begin{lemma}
  \label{hoch:existence} 
In the notation of Construction \emph{\ref{hoch:comparison}}, the assignment
  \[
 (L,L')\mapsto
\Hom{A\otimes_K{A'}}S{\injres_{A\otimes_K{A'}}(\projres_A(L)\otimes_K\projres_{A'}(L'))}\,,
  \]
defines a functor
  \[ 
\Rhom{A\otimes_K{A'}}S{-\dtensor K-}\col\dcat\times\dcat[\op
S]\to\dcat[S^{\sfc}]\,,
  \]
and the assignment
  \[
(\lambda,\lambda')\mapsto[\omega,\omega'](\lambda,\lambda')\,,
  \]
given by formula \eqref{eq:comparison}, defines a canonical natural 
equivalence of functors
  \[
[\omega,\omega']\col \Rhom{A\otimes_K{A'}}S{-\dtensor
K-}\to\Rhom{B\otimes_K{B'}}S{-\dtensor K-}\,.
  \]
If $K\to C\xra{\gamma}S$ and $K\to C'\xra{\gamma'}S$ are flat DG algebra
resolutions of $\sigma$ and $\op\sigma$, respectively, and $\vartheta\col
B\to C$ and $\vartheta'\col B'\to C'$ are morphism of resolutions, then
 \[
[\vartheta\omega,\vartheta'\omega']=
[\vartheta,\vartheta'][\omega,\omega']\,.
  \]
\end{lemma}

  \begin{proof}
Recall that the maps $E\xra{\epsilon}I\xla{\iota}J$ are unique up to homotopy.  
Thus, $\Hom{A\otimes_K{A'}}S{\epsilon}$ and $\Hom{\omega\otimes_K{\omega'}}S{\iota}$ 
are morphisms of complexes of $S^{\sfc}$-modules defined uniquely up to homotopy.  In view of
\eqref{eq:comparison}, this uniqueness has the following consequences:

The morphism $[\omega,\omega'](\lambda,\lambda')$ depends only on
$\lambda$ and $\lambda'$; one has
  \[
[\id^A,\id^{A'}](\id^L,\id^{L'})
=\id^{\Rhom{A\otimes_K{A'}}S{L\dtensor K{L'}}}\,;
  \]
and for all morphism $\mu\col M\to N$ and $\mu'\col M'\to N'$ of 
complexes of $S$-modules and $\op S$-modules, respectively, there are
equalities
  \[
[\vartheta\omega,\vartheta'\omega'](\mu\lambda,\mu'\lambda')=
[\vartheta,\vartheta'](\mu,\mu')\circ[\omega,\omega'](\lambda,\lambda')\,.
  \]

Suitable specializations of these properties show that
$\Rhom{A\otimes_K{A'}}S{-\dtensor K-}$ is a functor to
$\dcat[S^{\sfc}]$ from the product of the categories of complexes
over $S$ with that of complexes over $\op S$, \emph{and} that
$[\omega,\omega']$ is a natural transformation.

To prove that $[\omega,\omega']$ is an equivalence, it suffices to
show that if $\lambda$ and $\lambda'$ are quasi-isomorphisms, then
$\Rhom{\omega\otimes_K{\omega'}}S{\lambda\dtensor K\lambda'}$ is an
isomorphism.

By \eqref{eq:comparison}, it is enough to show that
$\Rhom{A\otimes_K{A'}}S{\epsilon}$ is a quasi-isomorphism.  As
$\lambda$ and $\lambda'$ are quasi-isomorphisms, the diagrams in
\eqref{eq:twodiagrams} imply that so are $\wt\lambda$ and $\wt\lambda'$.
Due to the $K$-flatness of $A\nat$ and $B'\nat$, their semiprojective
DG modules are $K$-flat, hence $\wt\lambda\otimes_K\wt\lambda'$
is a quasi-isomorphism of DG modules over $A\otimes_KA'$.  Now
diagram \eqref{eq:onediagram} shows that $\epsilon\col E\to I$ is a
quasi-isomorphism.  It follows that it is a homotopy equivalence, because
both $E$ and $J$ are semiinjective DG modules over $A\otimes_KA'$.  This
implies that $\Hom{A\otimes_K{A'}}S{\epsilon}$ is a quasi-isomorphism,
as desired.
  \end{proof}

To clarify how the natural equivalence in
Lemmas~\ref{hoch:existence} depends on $\omega$, we apply Quillen's
homotopical approach in~\cite{Qu}. It is made available by the following
result, see Baues and 
Pirashvili \cite[A.3.1, A.3.5]{BP}:

\begin{chunk}
 \label{DGAres:category}
The category of DG $K$-algebras has a model structure, where
  \begin{enumerate}[\rm\quad(1)]
    \item[$\bullet$]
the \emph{weak equivalences} are the quasi-isomorphisms;
    \item[$\bullet$]
the \emph{fibrations} are the morphisms that are surjective in positive 
degrees;
    \item[$\bullet$]
any DG $K$-algebra, whose underlying graded algebra is the tensor algebra
of a non-negatively graded projective $K$-module, is \emph{cofibrant};
that is, the structure map from $K$ is a \emph{cofibration}.
  \end{enumerate}

We recall some consequences of the existence of a model
structure, following \cite{DS}:  For all DG $K$-algebras $T$ and $A$, 
there exists a relation on the set of morphisms $T\to A$, known 
as \emph{left homotopy}, see \cite[4.2]{DS}.  It is an equivalence 
when $T$ is cofibrant, see~\cite[4.7]{DS}, and then $\pi^{\ell}(T,A)$ 
denotes the set of equivalence classes.  
 \end{chunk}

  \begin{lemma}
    \label{hoch:uniqueness}
There is a DG algebra resolution $K\to T\to S$ of $\sigma$ with 
$T$ cofibrant.

If $K\to A\xra{\alpha} S$ is a flat DG algebra resolutions of~$\sigma$, then there 
is a morphism of resolutions $\omega\col T\to A$.  Any morphism 
of resolutions $\varpi\col T\to A$ is left homotopic to $\omega$, and 
the natural equivalences defined in Lemma 
\emph{\ref{hoch:existence}} satisfy
  \[
[\omega,\op\omega]=[\varpi,\op\varpi]\col
\Rhom{\env T}S{-\dtensor K-}\to\Rhom{\env A}S{-\dtensor K-}
  \]
   \end{lemma}

\begin{proof}
Being both a fibration and a weak equivalence, $\alpha$ is, by definition,
an \emph{acyclic fibration}.  The existence of $\omega$ comes
from a defining property of model categories---the \emph{left
lifting property} of cofibrations with respect to acyclic fibrations;
see axiom MC4(i) in \cite[3.3]{DS}.  Composition with $\alpha$ induces
a bijection $\pi^{\ell}(T,A)\to\pi^{\ell}(T,S)$, see \cite[4.9]{DS},
so $\alpha\varpi=\alpha\omega$ implies that $\varpi$ and $\omega$ are
left homotopic.

By \cite[4.3,\,4.4]{DS}, the homotopy relation produces a 
commutative diagram
   \[
 \xymatrixrowsep{1.3pc}\xymatrixcolsep{3pc}\xymatrix{ 
& T \ar@{->}[d]^{\iota}\ar@/^1pc/[dr]^-{\omega}
   \\
\ T \ar@{=}[ur]\ar@{=}[dr] 
& \ar@{->}[l]_{\rho}^-{\simeq}C\ar@{->}[r]^-{\chi}
& A
   \\
& T \ar@{->}[u]_{\iota'}\ar@/_1pc/[ur]_-{\varpi}
}
   \]
of DG $K$-algebras, with a quasi-isomorphism $\rho$.  It induces a 
commutative diagram
 \[
 \xymatrixrowsep{2pc}\xymatrixcolsep{3pc}\xymatrix{
& T\otimes_K\op T
 \ar@{->}[d]^{\iota\otimes_K\op T}\ar@/^1.5pc/[dr]^-{\omega\otimes_K\op T}
   \\
T\otimes_K\op T \ar@{=}[ur]\ar@{=}[dr] 
 & C\otimes_K\op T
\ar@{->}[r]^-{\chi\otimes_K\op T}
\ar@{->}[l]_-{\rho\otimes_K\op T}^-{\simeq}
 & A\otimes_K\op T
   \\
& T\otimes_K\op T
\ar@{->}[u]_{\iota'\otimes_K\op T}
\ar@/_1.5pc/[ur]_-{\varpi\otimes_K\op T}
 }
   \]
of morphisms of DG $K$-algebras, where $\rho\otimes_K\op T$ is a
quasi-isomorphism because $\op T$ is $K$-flat.  The diagram above 
yields the following chain of equalities:
  \begin{align*}
[\omega,\id^{\op T}]
=[\chi,\id^{\op T}][\iota,\id^{\op T}]
=[\chi,\id^{\op T}][\rho,\id^{\op T}]^{-1}
=[\chi,\id^{\op T}][\iota',\id^{\op T}]
=[\varpi,\id^{\op T}].
 \end{align*}

A similar argument shows that the morphisms
$\op\omega$ and $\op\varpi$ are left homotopic, and yields
$[\id^A,\op\omega]=[\id^A,\op{\varpi}]$. Assembling these
data, one obtains
 \[
[\omega,\op\omega]
= [\id^A,\op\omega][\omega,\id^{\op T}]
= [\id^A,\op{\varpi}][\varpi,\id^{\op T}]
=[\varpi,\op\varpi]\,.  \qedhere
 \]
\end{proof}

 \stepcounter{theorem}
  \begin{proof}[Proof of Theorem~\emph{\ref{hoch:main}}]
Choose a DG algebra resolution $K\to T\to S$ of $\sigma$ with $T$
cofibrant, either by noting that the one in \ref{hoch:res} has this
property by \ref{DGAres:category}, or referring to a defining property of 
model categories; see axiom MC5(i) in \cite[3.3]{DS}.

For each flat DG algebra resolution $K\to A\to S$ of $\sigma$, form the
flat DG algebra resolution $K\to\op A\to\op S$ of $\op\sigma$, and 
define a functor
  \begin{equation}
    \label{eq:derivedHoch}
\Rhom{\env A}S{-\dtensor K-}\col
\dcat\times\dcat[\op S]\to\dcat[S^{\sfc}]
  \end{equation}
by applying Lemma \ref{hoch:existence} with $A'=\op A$.  As $T$ is
cofibrant, Lemma \ref{hoch:uniqueness} provides a morphism of resolutions
$\omega\col T\to A$, and shows that it defines a natural equivalence
  \[
[\omega,\op\omega]\col 
\Rhom{\env T}S{-\dtensor K-}\to\Rhom{\env A}S{-\dtensor K-}\,;
  \]
that does not depend on the choice of $\omega$; set
$\omega^{A}_T=[\omega,\op\omega]$.

When $K\to U\to S$ also is a flat DG algebra resolution of $\sigma$
with $U$ cofibrant, one gets morphisms of resolutions $\tau\col T\to U$
and $\theta\col U\to A$.  Both $\theta\tau\col T\to A$ and $\omega$
are morphisms of resolutions, so Lemmas \ref{hoch:uniqueness} and
\ref{hoch:existence} yield
  \[
\omega^{A}_T=[\omega,\op\omega]
=[\theta\tau,\op\theta\op\tau]
=[\theta,\op\theta][\tau,\op\tau]
=\omega^{A}_U\omega^{U}_T\,.
  \]

For each flat DG algebra resolution $K\to B\to S$ of $\sigma$ set
  \begin{equation}
    \label{eq:nattransf2}
\omega^{AB}:=\omega^{B}_T(\omega^{A}_T)^{-1}\col
\Rhom{\env A}S{-\dtensor K-}\to\Rhom{\env B}S{-\dtensor K-}\,.
  \end{equation}
One clearly has $\omega^{AC}=\omega^{BC}\omega^{AB}$, and 
$\omega^{AB}$ is independent of $T$, because
  \[
\omega^{B}_T(\omega^{A}_T)^{-1}
=\omega^{B}_U\omega^{U}_T(\omega^{A}_U\omega^{U}_T)^{-1}
=\omega^{B}_U\omega^{U}_T(\omega^{U}_T)^{-1}(\omega^{A}_U)^{-1}
=\omega^{B}_U(\omega^{A}_U)^{-1}\,.
  \]
It follows that $\omega^{AB}$ is the desired canonical natural equivalence.
  \end{proof}

We proceed with a short discussion of other derived Hochschild functors.
The proof of the next result is omitted, as it parallels that of
Theorem \ref{hoch:main}. 

  \begin{theorem}
   \pushQED{\qed}
    \label{hoch:main2}
Any flat DG algebra resolution $K\to A\to S$ of $\sigma$  
defines a functor 
\[
A\otimes_{\env A} {\Rhom K--}
\col\dcat^{\mathsf{op}}\times\dcat\to\dcat[S^{\sfc}]\,.
\]
For each flat DG algebra resolution $K\to B\to S$ of $\sigma$ one has a 
canonical equivalence
  \[
\omega_{BA}\col B\otimes_{\env B}{\Rhom K--}\xra{\,\simeq\,}
A\otimes_{\env A}{\Rhom K--}
  \]
of functors, and every flat DG algebra resolution $K\to C\to S$ of
$\sigma$ satisfies
  \[
\omega_{CA}=\omega_{BA}\omega_{CB}\,.
 \qedhere
  \]
    \end{theorem}

  \begin{remark}
   \label{hoch:homology}
We fix a DG algebra resolution $K\to A\to S$ of $\sigma$ and let
  \[
S\dtensor{S\dtensor K\op S}{\Rhom K--}\col\dcat^{\mathsf{op}}\times\dcat
\to\dcat[S^{\sfc}]
  \]
denote the functor $A\otimes_{A\otimes_{K}A}{\Rhom K--}$: The preceding
theorem shows that it is independent of the choice of $A$.  For all
$M,N\in\dcat$ it defines \emph{derived Hochschild homology modules}
of the $K$-algebra $S$ with Hom-decomposable coefficients:
  \[
\Tor n{S\dtensor K\op S}S{\Rhom KMN} =\CH n{S\dtensor{S\dtensor K\op
S}{\Rhom KMN}}\,.
  \]
These modules are related to classical Hochschild homology: 

For all $S$-modules $M$ and $N$ there are canonical natural maps
 \[
\Tor n{S\otimes_{K}\op S}S{\Hom KMN}\to\hch nKS{\Hom KMN}
  \]
of $S^{\sfc}$-modules, where the modules on the left are the classical 
ones, see \ref{DGoverRings}. They are isomorphisms when $S$ is $K$-flat;
see \cite[IX, \S6]{CE}. When $M$ is $K$-projective there exist natural homomorphisms 
 \begin{align*}
\Tor n{\alpha\otimes_K\op\alpha}S{\Rhom KMN}\col\qquad\qquad&{\ }\\
\Tor n{S\dtensor K\op S}S{\Rhom KMN}&\to
\Tor n{S\otimes_K\op S}S{\Hom KMN}
  \end{align*}
When $S$ is $K$-flat the composition $K\to S\xra{=}S$ is a flat DG resolution 
of $\sigma$ and $\alpha\col A\to S$ is a morphism of resolutions, so the 
theorem shows that the maps above are isomorphisms.
   \end{remark}

The remaining two composed functors collapse in a predictable way. 

\begin{remark}
  \label{DGAres:prereduction}
Similarly to Theorems \ref{hoch:main} and \ref{hoch:main2}, one can 
define functors
  \begin{align*}
\Rhom{S\dtensor K{\op S}}S{\Rhom K--}&\col
\dcat^{\mathsf{op}}\times\dcat\to\dcat[S^{\sfc}]\,,
  \\
S\otimes_{S\dtensor K{\op S}}{(-\dtensor K-)}&\col\dcat\times\dcat
\to\dcat[S^{\sfc}]\,,
  \end{align*}
that do not depend on the choice of the DG algebra
resolution $A$.  However, this is not necessary, as for all 
$M,N\in\dcat$ there exit canonical isomorphisms
  \begin{align}
     \label{eq:prereduction11}
\Rhom{S\dtensor K\op S}S{\Rhom KMN}&\simeq\Rhom SMN\,,
 \\
     \label{eq:prereduction21}
S\otimes_{S\dtensor K{\op S}}{(M\dtensor KN)}&\simeq M\dtensor SN\,.
  \end{align}
They are derived extensions of classical reduction results, 
\cite[IX.2.8, IX.2.8a]{CE}.
  \end{remark}

We finish with a comparison of the content of this section and that of
\cite[\S2]{YZ1}.

  \begin{remark}
    \label{rem:YZ}
When $M=N$ the statement of Theorem~\ref{hoch:main} bears a close
resemblance to results of Yekutieli and Zhang, see \cite[2.2, 2.3]{YZ1}.
One might ask whether their proof can be adapted to handle the general
case.

Unfortunately, even in the special case above the argument for
\cite[Theorem 2.2]{YZ1} is deficient.  It utilizes the mapping cylinder
of morphisms $\phi_0,\phi_1\col\tilde M\to M$ of DG modules over a
DG algebra,~$\tilde B$.  On page 3225, line 11, they are described
as ``the two $\tilde B'$-linear quasi-isomorphisms $\phi_0$ and
$\phi_1$'' where $\tilde B'$ is a DG algebra equipped with \emph{two}
homomorphisms of DG algebras $u_0,u_1\col\tilde B'\to\tilde B$; with
this, an implicit choice is being made between $u_0$ and $u_1$.  Such a choice 
compromises the argument, whose goal is to establish an equality 
$\chi_0=\chi_1$ between morphism of complexes $\chi_i$, which have
already been constructed by using $\phi_i$ and $u_i$ for $i=0,1$.

The basic problem is that the relation between various choices
of comparison \emph{morphisms of DG algebra resolutions} is not 
registered in the additive environment of derived categories. 
In the proof of Theorem \ref{hoch:main} it is
solved by using the homotopy equivalence provided by a model structure
on the category of DG algebras. 
  \end{remark}

\section{Reduction of derived Hochschild functors over algebras}
\label{Reduction over algebras}

Let $\sigma\col K\to S$ be a homomorphism of \emph{commutative} rings.  

Recall that $\sigma$ is said to be \emph{essentially of finite type} 
if it can be factored as 
  \[
K\hra K[x_1,\dots,x_d]\to V^{-1}K[x_1,\dots,x_d]\tra S\,,
  \]
where $x_1,\dots,x_d$ are indeterminates, $V$ is a multiplicatively
closed subset, the first two maps are canonical, and the third one is
a surjective ring homomorphism.

The following theorem, which is the main algebraic result in the paper,
involves the relative dualizing complex $D^\sigma$ described in
\eqref{eq:reldual}.

\begin{theorem}
\label{thm:reduction2}
If $\fd_KS$ is finite, then in $\dcat$ there are isomorphisms
 \begin{align}
  \label{eq:reduction21}
\Rhom{S\dtensor{K}S}S{M\dtensor KN}
&\simeq\Rhom S{\Rhom SM{D^{\sigma}}}N
  \\
  \label{eq:reduction22}
\Rhom{S\dtensor{K}S}S{\Rhom SM{D^\sigma}\dtensor KN}
&\simeq\Rhom SMN
            \end{align}
for all $M\in\dcatf{\sigma}$ and $N\in\dcat$; these morphisms are natural 
in $M$ and $N$. 
 \end{theorem}

We record a useful special case, obtained by combining Theorems
\ref{thm:reduction2} and \ref{thm:uniqueness}:

\begin{corollary}
Assume that $\sigma$ is flat, and let $K\to P\to S$ be a factorization
of $\sigma$ with $K\to P$ essentially smooth of relative dimension $d$
and $P\to S$ finite.

If $M$ is a finite $S$-module that is flat over $K$, and $N$ is an 
$S$-module, then for each $n\in\BZ$ there is an isomorphism 
of $S$-modules
\begin{xxalignat}{3}
&{\phantom{\square}}
&\Ext n{S\otimes_KS}S{M\otimes_KN}&\cong
\Ext{n-d}S{\Rhom PM{\Omega_{P|K}^d}}N\,.
&&\qed
\end{xxalignat}
\end{corollary}

Before the proof of Theorem~\ref{thm:reduction2} we make a couple of 
remarks.

  \begin{chunk}
    \label{crossover}
For all complexes of $P$-modules $L$, $X$, and $J$ there is a
natural morphism
  \[
\Hom P{L}P\otimes_PX\otimes_PJ\lra
\Hom P{\Hom P{X}{L}}{J}
   \]
defined by the assignment $\lambda\otimes x\otimes j\mapsto
\big(\chi\mapsto(-1)^{(|x|+|j|)|\lambda|}\lambda\chi(x)j\big)$.
This morphism is bijective when $L\nat$ and $X\nat$ are finite 
projective:  This is clear when $L$ and $X$ are shifts of $P$.
The case when they are shifts of projective modules follows, as the 
functors involved commute with finite direct sums.  The general case
is obtained by induction on the number of the degrees in 
which $L$ and $X$ are not zero.  
   \end{chunk}

\begin{chunk}
 \label{CDGA}
A DG algebra $A$ is called \emph{graded-commutative} if $ab=(-1)^{|a||b|}ba$ 
holds for all $a,b\in A$.  The identity map $\op A\to A$ then is a morphism of DG 
algebras, so each DG $A$-module is canonically a DG module over $\op A$, and
for all $A$-modules $M$ and $N$ the complexes $\Rhom AMN$ and $M\dtensor AN$ 
are canonically DG $A$-modules.  

When $A$ and $B$ are graded-commutative DG algebras, then so is $A\otimes_KB$, 
and the canonical isomorphisms in \eqref{eq:isos1} and \eqref{eq:isos2} represent 
morphisms in $\dcat[A\otimes_KB]$.
  \end{chunk}

\begin{proof}[Proof of  Theorem \emph{\ref{thm:reduction2}}]
The argument proceeds in several steps, with notation introduced 
as needed.  It uses chains of quasi-isomorphisms that involve a 
number of auxiliary DG algebras and DG modules.  We start 
with the DG algebras.

\begin{step}
 \label{chooseDiagram}
There exists a commutative diagram of morphisms of DG $K$-algebras
\begin{equation*}
 \begin{gathered}
   \xymatrixrowsep{3pc} 
   \xymatrixcolsep{.8pc} \xymatrix{
&&&&K \ar@{->}[dlll]\ar@{->}[drr]_-{\varkappa}\ar@{->}[rrrr]^-{\sigma}&&&& S
\\
&{\env P}\ar@{->}[d]_{\env\eta}\ar@{->}[drrr]_-{\iota}\ar@{->>}[rrrrr]_-{\mu} 
&&&&& P \ar@{->}[drr]_-{\eta}\ar@{->}[d]_(.65){\env\eta\otimes_{\env P}P}\ar@{->}[urr]_-{\pi}
    \\
&{\env B} \ar@{->>}[drrr]_(.35){\env B\otimes_{\env P}\iota} 
&&&A \ar@{->}[d]_(.25){\env\eta\otimes_{\env P}A}\ar@{->>}[urr]^(.35){\simeq}_(.25){\alpha}
&&{\env B}\otimes_{\env P}P \ar@{->}[drr]^(.37){\cong}_(.47){\nu}
&&B \ar@{->>}[uu]^-{\simeq}_-{\beta}
     \\
&&&& {\mathstrut^{\mathstrut}{\env B}}\otimes_{\env P}A
\ar@{->>}[urr]^(.67){\simeq}_(.55){\env B\otimes_{\env P}\alpha}\ar@{=}[r]
&C\ar@{->>}[rr]^-{\simeq}_-{\gamma} 
&&{\ov C}\ar@{=}[r]
&{\mathstrut^{\mathstrut} B}\otimes_PB\ar@{->>}[u]_{\mu'}}
 \end{gathered}
\end{equation*}
where $\simeq$ flags quasi-isomorphisms and $\twoheadrightarrow$ tips surjections.
The morphisms appearing in the diagram are constructed in the following sequence:
 \end{step}

Fix  a factorization $K\xra{\,\varkappa\,}P\xra{\,\pi\,}S$ of $\sigma$,
with $\varkappa$ essentially smooth of relative dimension $d$ and $\pi$
finite.

Set $\env P=P\otimes_KP$ and let $\mu\col\env P\to P$ denote the
multiplication map, and note that the projective dimension $\pd_{\env P}P$
is finite by \ref{smooth:criteria}.

Choose a graded-commutative DG
algebra resolution $\env P\xra{\iota}A\xra{\alpha}P$ of $\mu$ with
$A_0=\env P$, each $A_i$ a finite projective $\env P$-module, and $\sup
A=\pd_{\env P}P$; see \cite[2.2.8]{Av}.

Choose a graded-commutative DG algebra
resolution $P\xra{\eta}B\xra{\pi}S$ of $\sigma$, with $B_0$ a finite
free $P$-module and each $B_i$ a finite free $P$-module; again,
see \cite[2.2.8]{Av}.

Set $\env B=B\otimes_KB$ and let $\mu'\col B\otimes_PB\to B$ be 
the multiplication map.

Let $\nu\col{\env B}\otimes_{\env P}P\to B\otimes_PB$ be the map 
$b\otimes b'\otimes p\mapsto(b\otimes b')p$.

Let $\gamma\col{\env B}\otimes_{\env P}A\to B\otimes_PB$ be the 
map $b\otimes b'\otimes a\mapsto(b\otimes b')\alpha(a)$.

The diagram commutes by construction.   
The map $\nu$ is an isomorphism by \eqref{eq:isos2}, and $\env B\otimes_{\env P}\alpha$ 
is a quasi-isomorphism because $\alpha$ is one and $\env B$ is a bounded 
below complex of flat $\env P$-modules.

\medskip

We always specify the DG algebra operating on any newly introduced DG 
module.  On DG modules of homomorphisms and tensor products the
operations are those induced from the arguments of these functors; 
see \ref{CDGA}.

\begin{Notation}
Let $\env P\xra{\simeq} U$ be a semiinjective resolution over $\env P$.

Set $H=\hh{\Hom{\env P}PU}.$
\end{Notation}

\stepcounter{theorem}

\begin{step}
\label{homology}
There exists a unique isomorphism inducing $\id^{H}$ in homology:
 \begin{equation*}
H \simeq \Rhom{\env P}P{\env P}
  \quad\text{in }\dcat[P]\,.
\end{equation*}
\end{step}

 \begin{proof}
The isomorphism $H\cong\Ext{}{\env P}P{\env P}$ of graded $P$-modules
and \eqref{eq:smooth4} show that $H$ is projective, so Lemma 
\ref{homology:projective} applies.
 \end{proof}

\begin{Notation}
Set $L=\Hom P{H}P$.

Let $L\xra{\,\simeq\,}I$ be a semiinjective resolution over $P$.
\end{Notation}

\begin{step}
\label{balance}
There exists an isomorphism $D^\sigma\simeq\Rhom PSI$ in $\dcat$.
  \end{step}

\begin{proof}
Theorem \ref{thm:uniqueness} provides the first isomorphism in the chain
\begin{align*}
D^\sigma
&\simeq\Rhom PS{\Shift^{d}\Omega^d_{P|K}}\\
&\simeq\Rhom PS{\Rhom P{\Rhom{\env P}P{\env P}}P}\\
&\simeq\Rhom PS{\Rhom P{H}P}\\
&\simeq\Rhom PSI\,.
\end{align*}
The remaining ones come from \eqref{eq:smooth5}, Step \ref{homology}, 
and the resolution $L\simeq I$.
  \end{proof}

\begin{Notation}
Let $X'\xra{\,\simeq\,}M$ be a semiprojective resolution over $B$, with
$X'_i$ a finite projective $P$-module for each $i$ and $\inf{X'}=\inf{\hh
M}$, see \ref{semi:existence}; set $q=\pd_PM$, and note that $q$ is 
finite by Lemma \ref{smooth:fpd}.

Set $X=X'/X''$, where $X''_i=X'_i$ for $i>q$, $X''_{q}=\dd(X_{q+1})$,
and $X''_i=0$ for $i<q$.  It is easy to see that $X''$ is a DG
submodule of $X'$, so the canonical map $X'\to X$ is a surjective
quasi-isomorphism of DG $B$-modules.  Since $X'\xra{\,\simeq\,}M$ is a
semiprojective resolution over $P$, each $P$-module $X_i$ is projective;
see \cite[2.4.P]{AF:hd}.

Let $G\xra{\,\simeq\,}\Hom PX{{L}}$ be a semiprojective resolution over $B$.

Let $N\xra{\simeq}J$ be a semiinjective resolution over $B$.

Set $\ul J=\Hom BS{J}$.
\end{Notation}

\begin{step}
 \label{represents-right}
There exists an isomorphism
  \[
\Rhom S{\Rhom SM{D^\sigma}}N \simeq \Rhom B{\Rhom PM{L}}N
       \quad\text{in }\dcat[B]\,.
  \]
\end{step}

\begin{proof}
The map $N\xra{\simeq}J$ induces the vertical arrows in the commutative diagram
\[
\xymatrixrowsep{2pc} \xymatrixcolsep{1pc} \xymatrix{ 
 N \ar@{=}[r]\ar@{..>}[d] & \Hom BSN
 \ar@{->}[rrrr]^-{\cong}_-{\Hom B{\beta}N}\ar@{->}[d] &&&& \Hom BBN
 \ar@{->}[d]^-{\simeq}
 \\
{\ul J}\ar@{=}[r] & \Hom BSJ \ar@{->}[rrrr]^-{\simeq}_-{\Hom B{\beta}J} &&&& \Hom BBJ}
\]
Note that $B$ acts on $N$ through $\beta$, which is surjective, so $\Hom B{\beta}N$ is
bijective.  The map $\Hom B{\beta}J$ is a quasi-isomorphism because $\beta$ is one and
$J$ is semiinjective.  By~\ref{semi:basechange}(2), $\ul J$ is semiinjective, so 
$N\to{\ul J}$ is a semiinjective resolution over $S$.
In the following chain of morphisms of DG $B$-modules the isomorphisms
are adjunctions:
\begin{align*}
 \Hom S{\Hom SM{\Hom PS{I}}}{\ul J}
 &\cong\Hom S{\Hom PM{I}}{\ul J}\\
 &=\Hom S{\Hom PM{I}}{\Hom BSJ}\\
 &\cong\Hom B{S\otimes_S\Hom PM{I}}{J}\\
 &\cong\Hom B{\Hom PM{I}}{J}\\
 &\simeq\Hom B{\Hom P{X'}{I}}{J}\\
 &\simeq\Hom B{\Hom PX{I}}{J}\\
 &\simeq\Hom B{\Hom PX{{L}}}{J}\\
 &\simeq\Hom B{G}{J}
\end{align*}
The quasi-isomorphisms are induced by $M\xla{\simeq}X'\xra{\simeq}X$,
$L\xra{\simeq}I$, and $G\xra{\simeq}X$, because $I$ is semiinjective
over $P$, $J$ is semiinjective over $B$, and $X$ is semiprojective 
over $P$.  The chain yields the desired isomorphism in $\dcat[B]$ as 
$\ul J$ is semi\-injective over $S$, $G$ is semiprojective over $B$, and 
Step \ref{balance} gives $\Hom PS{I}\simeq D^\sigma$.
  \end{proof}

\begin{Notation}
Let $F\xra{\simeq}B$ be a semiprojective resolution over $C$.
\end{Notation}

\begin{step}
\label{injective}
There exists an isomorphism 
  \[
\Rhom B{\Rhom PM{L}}N\simeq \Rhom CB{\Rhom P{\Rhom PM{L}}N}
\quad\text{in }\dcat[C]\,.
  \]
\end{step}

\begin{proof}
The DG ${\ov C}$-module ${\ov C}\otimes_CF$ is semiprojective by~\ref {semi:basechange}(1).
The map $F\xra{\simeq}B$ induces the vertical arrows in the commutative diagram
of DG $C$-modules
\[
\xymatrixrowsep{2pc} \xymatrixcolsep{1.3pc} \xymatrix{
F\ar@{->}[r]^-{\cong} \ar@{->}[d]_-{\simeq}
&C\otimes_{C}F\ar@{->}[rrr]^-{\simeq}_-{\gamma\otimes_{C}F} \ar@{->}[d]_-{\simeq}
 &&& {\ov C}\otimes_{C}F\ar@{->}[d] & \\
B\ar@{->}^-{\cong}[r] 
&C\otimes_{C}B\ar@{->}[rrr]^-{\cong}_-{\gamma\otimes_{C}B} &&& {\ov C}\otimes_{C}B
}
\]
where $\gamma\otimes_{C}B$ is an isomorphism because 
$\gamma $ is surjective and $C$ acts on $B$ through $\gamma$, and 
$\gamma\otimes_{C}F$ is a quasi-isomorphism because $\gamma$ is one and 
$F$ is semiprojective.

The resulting quasi-isomorphism ${\ov C}\otimes_CF\xra{\simeq}B$
induces the quasi-isomorphism in the following chain, because 
$\Hom P{G}{J}$ is semiinjective over ${\ov C}$ by~\ref{semi:basechange}(2):
\begin{align*}
 \Hom B{G}{J}
 &\cong\Hom{{\ov C}}B{\Hom P{G}{J}} \\
 &\simeq\Hom{{\ov C}}{{\ov C}\otimes_CF}{{\Hom P{G}{J}}}\\
 &\cong\Hom{C}{F}{{\Hom P{G}{J}}}\,.
\end{align*}
The first isomorphism reflects the action of ${\ov C}=B\otimes_PB$ on $\Hom P{G}{J}$,
the second one holds by adjunction.  The chain represents the desired isomorphism 
because $\Hom P{G}{J}$ is semiinjective over $C$; see \ref{CDGA}.
  \end{proof}

\begin{Notation}
Let $Y\xra{\simeq}N$ be a semiprojective resolution over $B$.
\end{Notation}

\begin{step}
\label{bidual}
There exists an isomorphism 
\[
\Rhom P{\Rhom PM{L}}N \simeq 
\Rhom{\env P}P{\env P}\dtensor{\env P}{(M\dtensor KN)}
\quad\text{in }\dcat[\env B]\,. 
\]
\end{step}

\begin{proof}
{}From $G\xra{\simeq}\Hom PX{L}$ one gets the first link in the chain
\begin{align*}
 \Hom P{G}J
 &\simeq \Hom P{\Hom PX{{L}}}{J}\\
 &\cong \Hom P{L}P\otimes_PX\otimes_PJ\\
 &= H\otimes_PX\otimes_PJ\\
 &\simeq H\otimes_PX\otimes_PY\\
 &\cong H\otimes_{\env P}(X\otimes_KY)\\
 &\simeq \Hom{\env P}PU\otimes_{\env P}(X\otimes_KY)
\end{align*}
of morphisms of DG ${\ov C}$-modules; it is a quasi-isomorphism
because the semiinjective DG $B$-module $J$ is semiinjective over 
$P$, see~\ref{semi:basechange}(4). 

The equality reflects the definition of $L$.

The composition $Y\xra{\simeq}N\xra{\simeq}J$ induces the third link; which is 
a quasi-isomorphism because $H$ and $X$ are semiprojective over $P$.  

The second isomorphism holds by associativity of tensor products; 
see \ref{eq:isos2}.

The quasi-isomorphism $H\simeq\Hom{\env P}{A}{\env P}$ 
from Step \ref{homology} induces the last link, which is a
quasi-isomorphism because $X\otimes_KY$ is semiflat over $\env P$.

Finally, the semiinjectivity of $\ov J$ and the semiflatness of 
$X\otimes_KY$ imply that the chain above represents the desired 
isomorphism in $\dcat[\env B]$.
  \end{proof}

\begin{step}
\label{theta}
There exists an isomorphism 
  \[
\Rhom{\env P}P{\env P}\dtensor{\env P}{(M\dtensor KN)} \simeq 
\Rhom{\env P}P{M\dtensor KN}
\quad\text{in }\dcat[\env B]\,. 
\]
\end{step}

\begin{proof}
The resolutions $A\xra{\simeq}P\xra{\simeq}U$ over $\env P$ induce 
quasi-isomorphisms
\[
   \Hom{\env P}{P}{U}\simeq
   \Hom{\env P}{A}{U}\simeq
   \Hom{\env P}{A}{\env P}
\]
of complexes of $\env P$-modules, which in turn induce a quasi-isomorphism
 \[
   \Hom{\env P}{P}{U}\otimes_{\env P}(X\otimes_KY)\simeq
   \Hom{\env P}{A}{\env P}\otimes_{\env P}(X\otimes_KY)
\]
of DG $\env B$-modules.  To wrap things up, we use the canonical 
evaluation morphism
\[
\Hom{\env P}{A}{\env P}\otimes_{\env P}(X\otimes_KY)
\to\Hom{\env P}{A}{X\otimes_KY}
\]
given by $\lambda\otimes x\otimes y\mapsto
\big(a\mapsto(-1)^{(|x|+|y|)|a|}\lambda(a)(x\otimes y)\big)$;
it is bijective, because the DG algebra $A$ is a bounded complex of 
finite projective $\env P$-modules.
 \end{proof}

\begin{Notation}
Let $X\otimes_K Y\xra{\simeq}V$ be a semiinjective resolution over $\env B$.
\end{Notation}

\begin{step}
\label{represents-left}
There exists an isomorphism
\[
\Rhom CB{\Rhom{\env P}P{M\dtensor KN}}\simeq  \Rhom{\env B}S{M\dtensor KN}
\quad\text{in }\dcat[\env B]\,. 
\]
\end{step}

\begin{proof}
The isomorphisms below come from adjunction formulas, see \eqref{eq:isos1}:
\begin{align*}
 \Hom{C}{F}{\Hom{\env P}{A}{X\otimes_KY}}
 &\cong\Hom{\env B}{{F}\otimes_{A}{A}}{X\otimes_KY}\\
 &\cong\Hom{\env B}{F}{X\otimes_KY}\\
 &\simeq\Hom{\env B}{F}V\\
 &\simeq\Hom{\env B}SV
\end{align*}
The quasi-isomorphisms are induced by $X\otimes_KY \simeq V$ and
$F\simeq S$, respectively, because $F$ is semiprojective over $\env B$ 
and $V$ is semiinjective over $\env B$.
\end{proof}

  \begin{step}
 \label{linearity}
The composed morphism of the chain of isomorphisms
   \begin{align*}
\Rhom S{\Rhom S{M}{D^\sigma}}N
&\simeq\Rhom B{\Rhom PM{L}}N\\
&\simeq\Rhom CB{\Rhom P{\Rhom PM{L}}N}\\
&\simeq\Rhom CB{\Rhom{\env P}P{\env P}\dtensor{\env P}{(M\dtensor KN)}}\\
&\simeq\Rhom CB{\Rhom{\env P}P{M\dtensor KN}}\\
&\simeq\Rhom{\env B}S{M\dtensor KN}\\
&\simeq\Rhom{S\dtensor{K}S}S{M\dtensor KN}
 \end{align*}
provided by Steps \ref{represents-right} through \ref{represents-left} and 
Theorem \ref{hoch:main}, defines an isomorphism in $\dcat$.
  \end{step}

 \begin{proof}
The diagram of DG algebras in Step \ref{chooseDiagram} provides a morphism 
from $\env B$ to every DG algebra appearing in the chain of canonical 
isomorphisms above.  Thus, each isomorphism in the chain above defines a
unique isomorphism in $\dcat[\env B]$.  Its source and target are complexes 
of \text{$S$-modules}, on which $\env B$ acts through the composed morphism of 
DG algebras $\env B\to B\to S$.  This map is equal to the composition 
$\env B\to\env S\to S$.  Therefore, Lemma \ref{quisms}, applied first to 
the quasi-isomorphism $\env B\to\env S$, then to the homomorphisms 
$S\to\env S\to S$ given by $s\mapsto s \otimes1$ and $s\otimes s'\mapsto ss'$, 
shows that the complexes above are also isomorphic in $\dcat$.
  \end{proof}

\begin{step}
\label{linearity2}
The morphism in Step \ref{linearity} is natural with respect to 
$M$ and $N$.
\end{step} 

\begin{proof}
The morphism in question is represented by a  composition 
of quasi-isomor\-phisms of DG modules over $\env B$, so it suffices 
to verify that each such quasi-isomor\-phism represents a natural
morphism in $\dcat[\env B]$.  

Three kinds of quasi-isomorphisms are used.  The one chosen in 
Step \ref{homology} involves neither $M$ nor $N$, and so works 
simultaneously for all complexes of $S$-modules; thus, no issues of
naturality arises there.  Some of the constituent quasi-isomorphisms 
themselves are natural isomorphisms, such as Hom-tensor adjunction 
or associativity of tensor products.  Finally, there are quasi-isomorphisms
of functors induced replacing some DG module with a 
semiprojective or a semiinjective resolution.  The induced morphism of 
derived functors are natural, because morphisms of DG modules define 
unique up to homotopy morphisms of their resolutions; see 
\ref{semi:existence}.
 \end{proof}

The isomorphism~\eqref{eq:reduction21} and its properties have now 
been established.  

Theorem~\ref{thm:reduction1}(1) shows that formula \eqref{eq:reduction22} 
is equivalent to~\eqref{eq:reduction21}.
  \end{proof}

The next result is an analog of Theorem~\ref{thm:reduction2} for the
derived Hochschild functor from Remark \ref{hoch:homology};
it can be proved along the same lines, so the argument is omitted.

\begin{theorem}
\label{thm:reduction4}
If $\fd_KS$ is finite, then in $\dcat$ there are isomorphisms
 \begin{align}
     \label{eq:reduction41}
S\dtensor {S\dtensor{K}S} {\Rhom KMN}
&\simeq \Rhom SM{D^{\sigma}}\dtensor SN\,,
  \\
     \label{eq:reduction42}
S\dtensor {S\dtensor{K}S}{\Rhom K{\Rhom SM{D^\sigma}}N}
&\simeq M\dtensor SN\,,
            \end{align}
for all $M\in\dcatf{\sigma}$ and $N\in\dcat$; this morphism is natural 
in $M$ and $N$. \qed
 \end{theorem}

Setting $M=S=N$ in \eqref{eq:reduction41} produces a remarkable expression for $D^{\sigma}$:

\begin{corollary}   \pushQED{\qed}
  \label{cor:reduction4}
In $\dcat$ there is an isomorphism
 \[
D^{\sigma}\simeq S\dtensor{S\dtensor{K}S}{\Rhom KSS}\,.
    \qedhere
            \]
 \end{corollary}

\begin{remark}
\label{rem:reduction2}
The right hand sides of \eqref{eq:reduction22} and \eqref{eq:prereduction11} 
coincide, so one might wonder whether the induced isomorphism of the
derived Hochschild functors on the right hand side might be induced by 
an isomorphism of their coefficients:
  \[
\Rhom SM{D^\sigma}\dtensor KN\simeq\Rhom KMN\,.
  \]
To prove that no such isomorphism exists in general, it suffices to
consider the case when $K$ is a field, $S=K[x]$ a polynomial ring
over $K$, and $M=S=N$.  The factorization $K\to K[x]=K[x]$ gives 
$D^\sigma\simeq K[x]$, hence the left-hand side is isomorphic to 
$K[x]\otimes_KK[x]$.  On the other hand, the $\env S$-module 
$\Rhom K{K[x]}{K[x]}$ on the right hand side has an uncountable 
basis as a $K$-vector space.
 \end{remark}

\section{Global duality}
\label{A:dual}

We now reconsider a portion of the preceding results from a global point
of view.  The facts needed from Grothendieck duality theory for schemes
are summarized in this section, and the globalized results given in
the next.

While it is not difficult to show that the complexes and functors we
will deal with specialize over affine schemes to sheafifications of
similar things that have appeared earlier, the corresponding statement
for functorial maps between such objects is not so easy to establish,
and we will not be settling this issue here.  Indeed, giving concrete
descriptions of abstractly characterized functorial maps is one of the
major problems of duality theory.  \smallskip

\pagebreak[3]

\emph{Schemes are assumed throughout to be noetherian.}
\vskip2pt

A scheme-map $f \colon X \to Y$ is \emph{essentially of finite type\/}
if every~$y \in Y$ has an affine open neighborhood $V = \Spec(A)$
such that $f^{-1}V$ can be covered by finitely many affine open $U_i =
\Spec(C_i)$ such that the corresponding ring homomorphisms $A \to C_i$
are essentially of finite type.

If, moreover, each $C_i$ is a localization of $A$ (that is, a ring of fractions) 
and $A\to C_i$ is the canonical map, then we say that $f$ is \emph{localizing}.

The property ``essentially finite\kern.5pt-type'' behaves well with
respect to composition and base change:  if $f\colon X\to Y$ and
\mbox{$g\colon  Y\to Z$} are scheme-maps, and if both $f$ and $g$ are
essentially of finite type, then so is the composition~$gf$; if $gf$
and $g$ are  essentially of finite type then so is~$f$; and if $Y'\to Y$
is any scheme\kern.5pt-map then $ X'\!:= Y'\times_Y X$ is noetherian,
and the projection $ X'\to Y'$ is essentially of finite type.\looseness=-1

Similar statements hold with ``localizing" in place of ``essentially
finite\kern.5pt-type.''\vspace{.6pt}

If the scheme\kern.5pt-map $f$ is localizing and also
injective (as a set-map) then we say that
$f$ is a \emph{localizing immersion}.

A scheme\kern.5pt-map  is \emph{essentially smooth}, resp.~\emph{essentially
\'etale}, if it~is essentially of finite type and formally smooth,
resp.~formally \'etale \cite[\S17.1]{Gr}.

For example, any localizing map is essentially \'etale: this assertion, being
local (see  \cite[(17.1.6)]{Gr}), results from \cite[(17.1.2)]{Gr} and 
\cite[(19.10.3)(ii)]{Gr_0}.) \vspace{1pt}

\begin{remark}\label{nice}
In several places we will refer to proofs in \cite{Il3} which make
use of the fact that the diagonal of a smooth map is a quasi-regular
immersion. To ensure that those proofs apply here, note that  the
same property for essentially smooth maps is given by \cite[16.10.2,\,
16.9.4]{Gr}.
\end{remark}

In \cite[4.1]{Nk2}, extending  a compactification theorem of
Nagata,  it is shown that \emph{any\- essentially-finite-type separated
map $f$ of noetherian schemes factors\- as \mbox{$f=\bar{f\,}\!u$}
with $\bar{f\,}\!$ proper and $u$~a localizing immersion.}\vspace{1.5pt}

\begin{example}\label{example}(Local compactification.)
A map $f\colon X=\Spec S\to\Spec K=Y\<$ coming from an essentially
finite\kern.5pt-type homomorphism of noetherian rings $ K\to S$ factors
as\looseness=-1
$$
X\xrightarrow{j} Z \overset{\lift.85,i,}\hookrightarrow  \bar Z\xrightarrow{\pi}Y,
$$
where $Z$ is the Spec of a finitely-generated $K$-algebra $T$ of which
$S$ is a localization, $j$ being the corresponding map, where $i$ is
an open immersion, and where $\pi$ is a projective map, so that $\pi$
is proper and $ij\>$ is a localizing immersion.
\end{example}

In the rest of this section we review basic facts about Grothendieck
duality, referring to \cite{Lp2} and \cite{Nk} for details\vspace{1pt}.

\medskip

\emph{Henceforth all scheme-maps are assumed to be essentially of
finite type, and separated.}

\medskip

For a scheme $X\<$, $\D(X)$ is the derived category
of the category of $\OX$-modules, $\Dc(X)\subset\D(X)$
(resp.~$\Dqc(X)\subset\D(X))$ is the full subcategory whose objects
are the $\OX$-complexes with coherent (resp.~quasi-coherent)
homology modules, and $\D_{\<\bullet}^{\lift.95,\text
{\cmt\char'053},}$~(resp.~$\D_{\<\bullet}^{\lift.95,\text
{\cmt\char'055},})$ is the full subcategory of $\D_{\<\bullet}$
whose objects are the complexes $E\in \D_{\<\bullet}$ with
$H^n(E)\!:=H_{-n}(E)=0$ for all $n\ll 0$ (resp.~$n\gg0$).

\begin{chunk}
To any scheme\kern.5pt-map $f\colon X\to Y$ one associates
the right-derived direct-image functor~$\Rf\colon\Dqc(X)\to\Dqc(Y)$
and its \emph{left adjoint,} the left-derived
inverse-image functor $\bL f^*\colon\Dqc(Y)\to\Dqc(X)$
\cite[3.2.2,\,3.9.1,\,3.9.2]{Lp2}.

These functors interact with
the left-derived tensor product $\Otimes{}$ via a natural isomorphism
\begin{equation}\label{^* and tensor}
\bL f^*(E\Otimes{Y}F)\iso
\bL f^*\<\<E\Otimes{\<\<X} \bL f^*\<\<F
\qquad\big(E,F\in\D(Y)\big),
\end{equation}
see \cite[3.2.4]{Lp2}; via the functorial map
\begin{equation}\label{_* and tensor}
\Rf G\Otimes{Y}\Rf H \to \Rf (G\Otimes{\<\<X} H)
\qquad\big(G,H\in\D(X)\big)
\end{equation}
adjoint to the natural composite map
$$
\bL f^*(\Rf G\Otimes{Y}\R f_*H)\underset{\eqref{^* and tensor}}\iso
\bL f^*\Rf G\Otimes{\<\<X} \bL f^*\Rf H\longrightarrow
G\Otimes{\<\<X} H;
$$
and via the \emph{projection isomorphism}
\begin{equation}\label{projection}
\Rf F\Otimes{Y} G \iso \Rf(F\Otimes{\<\<X} \bL f^*G)
\qquad\big(F\in\Dqc(X),\;G\in\Dqc(Y)\big),
\end{equation}
defined qua map to be  the natural composition
$$
\Rf F\Otimes{Y}G\lra \Rf F\Otimes{Y}\Rf\bL f^*G\underset{\eqref{_* and tensor}}\lra
\Rf(F\Otimes{\<\<X} \bL f^*G).
$$
see \cite[3.9.4]{Lp2}.
 \end{chunk}

\begin{chunk}
\label{* and RHom}
Interactions with the derived (sheaf-)homomorphism functor $\RH{}$ occur via
natural bifunctorial maps
\begin{equation}\label{^* and Hom}
\bL f^*\RH_Y(E, F)\to \RH_{\<X}(\bL f^*\<\<E\<,\> \bL f^*\<\<F)
\qquad\big(E,F\in\D(Y)\big),
\end{equation}
\begin{equation}\label{_* and Hom}
\Rf\RH_{\<X}(E,F)\to\RH_Y(\Rf E\<,\>\Rf F)\qquad\big(E,F\in\D(X)\big),
\end{equation}
the former corresponding via \eqref{global adjunction} below to the composite map
$$
\bL f^*\RH_{\<X}(E,F)\Otimes{\<\<X}\bL f^* E\>\>\underset{\eqref{^* and tensor}^{-\<1}}\iso 
\>\>\bL f^*\big(\RH_{\<X}(E,F)\Otimes{Y}E\big)
\xra{\bL f^*\<\<\varepsilon} \bL f^* F,
$$
with $\varepsilon$ corresponding  via~\eqref{global adjunction} to 
the identity map of $\RH_Y(E,F)$; and
the latter corresponding  to the composite map
$$
\Rf\RH_{\<X}(E,F)\Otimes{Y}\Rf E\underset{\eqref{_* and tensor}}\lra \Rf\big(\RH_{\<X}(E,F)\Otimes{\<\<X}E\big)
\xra{\Rf\varepsilon} \Rf F.
$$

The map \eqref{^* and Hom} is an \emph{isomorphism} if $f$ is an open
immersion, or if $E\in\Dcmi(Y)$, $F\in\Dqcpl(Y)$ and $f$~has finite flat
dimension \cite[4.6.7]{Lp2}.
\end{chunk}

\begin{chunk}\label{adjunction}
The fundamental adjunction relation between the derived tensor  and derived
homomorphism functors is expressed by the standard
trifunctorial isomorphism
$$
\RH_{\<X}\<\big(A\Otimes{\<\<X}\mkern-2.5mu B,C\big)\iso\RH_{\<X}\<\big(A,\RH_{\<X}(B,C)\big)
\qquad\big(A, B, C\in\D(X)\big),
$$
see e.g., \cite[\S2.6]{Lp2}.
Application of the composite functor $\textup{H}^0\R\Gamma(X,-)$
to this isomorphism produces a canonical isomorphism\vspace{2pt}
\begin{equation}\label{global adjunction}
\operatorname{Hom}_{\mathsf D(X)}\!\big(A\Otimes{\<\<X}\mkern-2.5mu B,C\big)\iso
\operatorname{Hom}_{\mathsf D(X)}\!\big(A,\RH_{\<X}(B,C\>)\big)
\qquad\big(A, B, C\in\D(X)\big).
\end{equation}

{From} among the many resulting maps, we will need the functorial one
\begin{equation}\label{Hom and Tensor}
\RH_{\<X}( M\<,\> E)\Otimes{\<\<X}F\lra
\RH_{\<X}( M\<,\mkern1.5mu E\Otimes{\<\<X}F)
\qquad\big(M, E, F\in\D(X)\big),
\end{equation}
corresponding via ~\eqref{global adjunction} to the natural composite map
(with $\varepsilon$ as above):\vspace{1.5pt}
$$
\smash{(\RH_{\<X}( M\<,\> E)\Otimes{\<\<X}F) \Otimes{\<\<X} M\iso
(\RH_{\<X}( M\<,\> E)\Otimes{\<\<X}M) \Otimes{\<\<X} F
\xra{\varepsilon\Otimes{\<\<X}1} E\Otimes{\<\<X}F.}
$$

The map \eqref{Hom and Tensor} is an \emph{isomorphism} if the complex
$M$ is \emph{perfect} (see \S\ref{Reduction over schemes}).  Indeed,
the question is local on $X\<$, so one can assume that $M$ is a bounded
complex of finite\kern.5pt-rank free $\OX$-modules. The assertion is
then given by a simple induction---similar to the one in the second-last
paragraph in the proof of \cite[4.6.7]{Lp2}---on the number of degrees
in which $M$ doesn't vanish.

Similarly, the map \eqref{Hom and Tensor} is an isomorphism
if $F$ is perfect.
\end{chunk}

\begin{chunk}\label{fiber square}
For any commutative square of scheme\kern.5pt-maps
\[
\xymatrixrowsep{1pc}
\xymatrixcolsep{1pc}
\xymatrix{
X'\ar@{->}[rr]^-{v}
\ar@{->}[dd]_-{g}
&&
X
\ar@{->}[dd]^-{f}
\\
&\Xi
\\
Y'
\ar@{->}[rr]_-{u}
&&
Y
}
\]
one has the  map
$\theta_{\>\Xi}\colon \bL u^* \Rf\to \R g_*\bL v^*$
adjoint to the
natural composite map
$$
\Rf\longrightarrow\Rf\R v_*\bL v^*\iso
\R u_*\R g_*\bL v^*.
$$
When $\Xi$ is a \emph{fiber square} (which means that the map associated to
$\Xi$
is an isomorphism $X'\iso X\times_Y Y'$), and $u$ is \emph{flat}, then
$\theta_{\>\Xi}$ is an \emph{isomorphism}. In fact, for any fiber square
$\Xi$, $\>\theta_{\>\Xi}$ \emph{is an isomorphism\/ $\iff \Xi$ is
tor-independent} \cite[3.10.3]{Lp2}.
\end{chunk}

\begin{chunk}\label{twisted inverse}
Duality theory focuses on the \emph {twisted inverse-image pseudo\-functor}
$$
f^!\colon\Dqcpl(Y)\to\Dqcpl(X),
$$
where  ``pseudofunctoriality" (also known as \mbox{ ``2-functoriality"})
entails, in addition to functoriality,  a family of functorial
isomorphisms
$c^{}_{\< f,\mkern1.5mu g}\colon (gf)^!\iso f^!g^!$,\vspace{1pt} one for each composable
pair
$X\xrightarrow{\lift1.05,f,} Y\xrightarrow{\lift.7,g,} Z$, satisfying a natural  ``associativity"
property
vis-\`a-vis any composable triple, see, e.g., \cite[3.6.5]{Lp2}.

This pseudofunctor is uniquely determined up to isomorphism by the following
three properties:\vspace{1pt}

(i) If $f$ is essentially \'etale then $f^!$ is the usual restriction functor
$f^*$.

(ii) If $f$ is proper then $f^!$ is right-adjoint to $\Rf$ (which takes $\Dqcpl(X)$ into
$\Dqcpl(Y)$ \cite[(3.9.2)]{Lp2}).

(iii) Suppose there is given a fiber square $\Xi$ as above, with $f$ (hence
$g$) proper and $u$ (hence $v$) essentially \'etale.
Then  the functorial \emph{base-change map}
\begin{equation}\label{beta}
\beta_{\>\Xi}(F)\colon:v^*\<\<f^!F\to g^!\<u^*F\qquad\big(F\in\Dqcpl(Y)\big),
\end{equation}
defined to be  adjoint to the natural composition
$$
\R g_* v^*\mkern-2mu f^!F\underset{\lift1.2,\theta_{\>\Xi}^{-\<1},}{\iso}
u^*\Rf f^!F
\longrightarrow u^*\<\<F,\\[1.5pt]
$$
is identical with
the natural composite \emph{isomorphism}
$$
v^*\<\<f^!F= v^!\<f^!F
\iso (fv)^!F=(ug)^!F\iso
g^!\<u^!F  =g^!\<u^*F.
$$
For the existence of such a pseudofunctor, see \cite[\S 5.2]{Nk2}.
\end{chunk}

\vspace{1pt}

\pagebreak[3]

\begin{remarks}\label{twisted rem}
(a) If $f$ has finite flat dimension (in addition to
being proper), then \eqref{beta} is an \emph{isomorphism} for \emph{all}
$F\in\Dqc(Y)$---see \cite[4.7.4]{Lp2} and
\cite[1.2]{LN}.\vspace{1pt}

(b) Theorem 5.3 in \cite{Nk2} (as elaborated in
\cite[7.1.6]{Nk}) states  that, moreover, one can associate, in
an essentially unique way, to \emph{any} fiber square $\Xi$ with $u$
(hence~$v$) flat, a functorial isomorphism $\beta_{\>\Xi}$, agreeing
with \eqref{beta} when $f$ is proper, and with the natural isomorphism
$v^*\<\<f^*\iso g^*u^*$ when $f$ is essentially \'etale.\vspace{1pt}

(c) Let $f\colon X\to Y$ be essentially smooth, so that by
\cite[16.10.2]{Gr} the relative
differential sheaf $\>\Omega_f$  is locally free over $\OX$.
On any connected component $W$ of~$X$, the
rank of~$\>\Omega_f$ is a constant, denoted $d(W)$.
There is a \emph{functorial isomorphism}\vspace{1pt}
\begin{equation}\label{smooth!}
f^!\<F\iso
\text{\cms\char'006}^d\>\lift2,\bwedge{}^{\lift1.7,\!d,}_{\lift.4,\>\OX,},
\!(\Omega_f)\>\otimes_{\OX} f^* \<\<F\qquad \big(F\in\Dqc(Y)\big),\\[1pt]
\end{equation}
with $\text{\cms\char'006}^d\>\lift2,\bwedge{}^{\lift1.7,\!d,}_{\lift.4,\>\OX,}, \!(\Omega_f)$  the complex whose restriction to any  $W$
is $\text{\cms\char'006}^{d(W)}\lift2,\bwedge{}^{\lift1.7,\!d(W),}_{\lift.4,\>\mathcal O_W,}, \<\<\big(\Omega_f\big|_W\big)$.\vspace{3.5pt}

\pagebreak[3]
To prove this, one may assume that $X$ itself  is connected, and set
$d\!:= d(X)$.  Noting that the diagonal $\Delta\colon X\to X\times_YX$ is
defined locally by a regular sequence of length $d$ \cite[16.9.4]{Gr},
so that \mbox{$\Delta^!\mathcal O_{X\times_Y X}\Otimes{}\bL \Delta^*
G\cong\Delta^!G$} for all $G\in\Dqc(X\times_Y X)$ \cite[p.\,180, 7.3]{H},
one can imitate the proof of \cite[p.\,397, Theorem 3]{V'}, where, in
view of~(b) above, one can drop the properness condition and take $U=X$,
and where finiteness of Krull dimension is superfluous.
 \end{remarks}

In this connection, see also~\ref{globalized composition} below,
and \cite[\S2.2]{Co2}.

\begin{chunk}\label{sheaf duality}
The fact that $\beta_{\>\Xi}(F)$ in \eqref{beta}  is an isomorphism for
all $F\in\Dqcpl(Y)$ whenever $u$ is an open immersion and $f$ is proper, is shown in
\cite[\S4.6, part V\kern.5pt]{Lp2} to be equivalent to \emph{sheafified
duality,} which is that \emph{for any proper $f\colon X\to Y,$ and
any \mbox{$E\in\Dqc(X),$} $F\in\Dqcpl(Y),$  the natural composition,
in which the first map comes from \textup{\ref{_* and Hom},}}
 \begin{equation}\label{duality iso}
\Rf\mathcal Hom_X(E,\,f^!\<F)\to \R\mathcal Hom_Y(\Rf E,\,\Rf f^!\<F)\to
\R\mathcal Hom_Y(\R f_*E,F),
 \end{equation}
\emph{is an isomorphism}.

Moreover, if the proper map $f$ has finite flat dimension,
then sheafified duality holds for \emph{all} $F\in\Dqc(Y)$, see
\cite[4.7.4]{Lp2}.\vspace{1pt}

If $f$ is a \emph{finite} map,  the isomorphism \eqref{duality iso}
with $E=\OX$ determines the functor $f^!$ up to isomorphism. (See  \cite[\S2.2]{Co2}.)
In the affine case, for example, if~$f\colon\Spec B\to \Spec A$ corresponds
to a finite  ring homomorphism $A\to B$, and  $^\sim$~denotes sheafification, then for an $A$-complex $M\<$, the $B$-complex
$f^!(M^\sim)$ can be defined by the equality
\begin{equation}\label{finite}
f^!(M^\sim) = \R\textup{Hom}_A(B,M)^\sim.
\end{equation}
\end{chunk}

\begin{chunk}\label{globalized composition}($f^!$ and $\Otimes{}$).\vspace{1pt}
For any  $f=\bar{f\,}\!u$ with $\bar{f\,}\!$ proper and $u$ localizing,
and $E$, $F\in\Dqcpl(Y)$ such that  $E\Otimes{Y}F\in\Dqcpl(Y)$ \vspace{1pt}
(e.g., $E$ perfect, see \S\ref{Reduction over schemes}),
there is a canonical functorial map\looseness=-1
\begin{equation}\label{^! and Tensor}
f^!\<\<E\Otimes{\<\<X}\bL f^*\<\<F\to f^!(E\Otimes{Y}F)
\end{equation}
equal, when $u$=1, to the map $\chi^f$ adjoint to the natural composite map
$$
\Rf(f^!\<\<E\Otimes{\<\<X}\bL f^*\<\<F)
\iso \Rf f^!\<\<E\Otimes{Y}F \longrightarrow
E\Otimes{Y}F,
$$
(see \eqref{projection}), and equal,  in the general case,  to the natural composition
\pagebreak[3]
\begin{multline}\label{general case}
f^!\<\<E\Otimes{\<\<X}\bL f^*\<\<F\cong u^*\<\<\bar{f\,}\!^! \<\<E\Otimes{\<\<X}
u^*\bL \bar{f\,}\!^*\<\<F
\cong u^*(\bar{f\,}\!^! \<\<E\Otimes{\<\<X}\bL \bar{f\,}\!^*\<\<F)\\
\xrightarrow{u^*\<\<\chi^{\bar{f\,}\!\!\!}}
u^*\<\<\bar{f\,}\!^!(E\Otimes{Y}\<F)
\cong
f^!(E\Otimes{Y}\<F).
\end{multline}
``Canonicity" means  \eqref{general case} depends only on~$f\<$, not on
the factorization~$f=\bar{f\,}\!u$. This is shown by imitation of the
proof of \cite[4.9.2.2]{Lp2},\vspace{-2pt} after one notes that
for any composition $X\xrightarrow{i} X'\xrightarrow{v} Y'$ with $i$
a closed immersion and $v$ localizing, the induced map from~$X$ to
its schematic image in $Y'$ is localizing: the question being local,
this just means that for a multiplicative system $M$ in a ring $B$,
and a $B_M$-ideal~$J$ with inverse image $I$ in $B$, the natural map
$(B/I)_M\to B_M/J$ is bijective. (See also \cite[5.8]{Nk2}.)

\stepcounter{equation}\label{tensor iso}\noindent
{\bf \theequation.} By \cite[Theorem 5.9]{Nk2},
the map \eqref{^! and Tensor} is an \emph{isomorphism} if $f$ has finite flat dimension and $E=\mathcal O_Y$---hence more generally if $E$ is \emph{perfect,}
cf.~end of \S\ref{adjunction}.
In particular, for any $g\colon Y\to Z$ there is a natural isomorphism
$$
(gf)^!\mathcal O_Z\cong f^!g^!\mathcal O_Z\iso f^!\mathcal
O_Y\Otimes{\<\<X}\bL f^*\<g^!\mathcal O_Z.
$$
In combination with~\ref{twisted rem}(c), and \eqref{reldual} below, this appears to be a
globalization of~\cite[Theorem 8.6]{AIL1}. But it is by no
means clear (nor will we address the point further here) that for maps
of affine schemes the present isomorphism agrees  with the sheafification
of the one in~\emph{loc.cit.}

\end{chunk}

\begin{chunk}
\label{twist and Hom}($f^!$ and $\RH$).
Let $f\colon X\to Y$ be a scheme\kern.5pt-map, $E\in\Dcmi(Y)$, $F\in\Dqcpl(Y)$. \emph{There is a canonical isomorphism}
\begin{equation}
\label{twisted RHom}
f^!\RH_Y(E\<,F\>)\iso \RH_{\<X}(\bL f^*\<\<E\<,\>f^!\<F\>).
\end{equation}
Indeed,  by \cite[p.\,92, 3.3]{H}, $\RH_Y(E\<,F\>)\in\Dqcpl(Y)$, so 
$f^!\RH_Y(E\<,F\>)\in\Dqcpl(X)$; and furthermore, $f^!\<F\in\Dqcpl(X)$ and,
by  \cite[p.\,99, 44]{H}, $\bL f^*E\in\Dcmi(X)$, so that
$\RH_{\<X}(\bL f^*E\<,\>f^!F\>)
\in\Dqcpl(X)$. (Those proofs in \cite{H} which are ``left to the reader" 
use \cite[p.\,73, 7.3]{H}.) So when $f$ is \emph{proper} (the only case we'll need),
the  map \eqref{twisted RHom} and its inverse 
come out of  the  following composite functorial isomorphism, 
for any $G\in\Dqcpl(X)$---in particular, $G=f^!\RH_Y(E\<,F\>)$ or $G=\RH_{\<X}(\bL f^*E\<,\>f^!F\>)$:
\begin{xxalignat}{2}
\operatorname{Hom}_{\mathsf D(X)}\!\<\<\big(G,f^!\RH_Y(E\<,F\>)\big)
\iso&
\operatorname{Hom}_{\mathsf D(Y)}\!\<\<\big(\R f_*G,\RH_Y(E\<,F\>)\big)
&&{ \textup{by~\ref{twisted inverse}(ii)}}
\\
\iso&
\operatorname{Hom}_{\mathsf D(Y)}\!\<\<\big(\R f_*G\Otimes Y\!E,F\>)\big)
&&{ \textup{by~\eqref{global adjunction}}}
\\
\iso&
\operatorname{Hom}_{\mathsf D(Y)}\!\<\<\big(\R f_*(G\Otimes {\<\<X}\<\bL f^*\<\<E)\<,F\>)\big)
&&{ \textup{by~\eqref{projection}}}
\\
\iso&
\operatorname{Hom}_{\mathsf D(X)}\!\<\<\big(G\Otimes {\<\<X}\<\bL f^*\<\<E\<,f^!\<F\>)\big)
&&{ \textup{by~\ref{twisted inverse}(ii)}}
\\
\iso&
\operatorname{Hom}_{\mathsf D(X)}\!\<\<\big(G,\RH_{\<X}(\bL f^*\<\<E\<, \>f^!\<F\>)\big)
&&{ \textup{by~\eqref{global adjunction}}
}
\end{xxalignat}
(For the general case, one compactifies, and shows canonicity\dots) 
\end{chunk}

\section{Reduction of derived Hochschild functors over schemes}
\label{Reduction over schemes}

Terminology and assumptions remain as in the first part of
Section~\ref{A:dual}. Again, \emph{all schemes are assumed to be
noetherian, and all scheme\kern.5pt-maps to be essentially of finite type,
and separated.} \vspace{1.5pt}

An $\OX$-complex $M$ is  \emph{perfect} if $X$ can be covered by open
sets $U$ such that the restriction $M|_U$ is $\D(U)$-isomorphic to a
bounded complex of finite-rank locally free $\mathcal O_U$-modules.
For a scheme\kern.5pt-map $f\colon X\to Y\<$, with $f_0$ the map $f$
considered only as a map of topological spaces, and  $f_0^{-1}$ the left
adjoint of the direct image functor~$f_{0*}$ from sheaves of abelian
groups on $X$ to sheaves of abelian groups on~$Y\<$,  there~is a standard
way of making  $f_0^{-1}\mathcal O_Y$ into a sheaf of commutative rings
on~$X\<$, whose stalk at any point $x\in X$ is $\mathcal O_{Y,\>f(x)\>}$.
An $\OX$-complex $M$ is \emph{$f$-perfect} if \mbox{$M\in\Dc(X)$} and $M$
is isomorphic in the derived category of $f_0^{-1}\mathcal O_Y$-modules
to a bounded complex of flat $f_0^{-1}\mathcal O_Y$-modules. Perfection
is equivalent to $\id^{\<X}$-perfection, with $\id^{\<X}$ the identity
map of $X$ \cite[p.\,135, 5.8.1]{Il1}.

If $f$ factors as $X\xrightarrow{i}Z\xrightarrow{g} Y$  with $g$
essentially smooth and $i$ a closed immersion, then $M$ is $f$-perfect
if and only if $i_*M$ is $(\id^{\<Z}\<$-)perfect: the proof of
\mbox{\cite[pp.\,252, 4.4]{Il3}} applies here (see Remark~\ref{nice}).
Using \cite[p.\,242, 3.3]{Il3}, one sees that \mbox{$f$-perfection} is
local on $X$: $M$ is $f$-perfect if and only if every $x\in X$ has an
open neighborhood~$U$ such that $M|_U$ is $f|_U$-perfect. Note that, $f$
being a composite of essentially finite\kern.5pt-type maps, and hence
itself essentially of finite type, there is always such a $U$ for which
$f|U$ factors as (essentially smooth)$\,\circ\,$(closed immersion).

Let $\dcatf f$ be the full subcategory of $\D(X)$ whose
objects\vspace{.5pt} are all the $f$-perfect complexes; and let
$\dcatf{X}\!:=\dcatf{\id^{\lift.85,X,}}$ be the full subcategory of
perfect $\OX$-complexes.

If $f\colon X=\Spec S\to\Spec K=Y\<$ corresponds to a homomorphism of
noetherian rings $\sigma\colon K\to S$, then $\dcatf f$ is equivalent to
the category~$\dcatf{\sigma}$ of~\S\ref{Reduction over algebras}: in~view
of the standard equivalence, given by sheafification, between coherent
$S$-modules and coherent $\OX$-modules, this follows from \cite[p.\,168,
2.2.2.1 and p.\,242, 3.3]{Il3}.

The central result in this section is the following theorem.\vspace{-3pt}

\begin{theorem}
\label{thm:reduction3}
Consider a commutative diagram of scheme-maps
\[
\xymatrixrowsep{2pc}
\xymatrixcolsep{2pc}
\xymatrix{
&&X
\ar@{->}[dr]^-{f}
\\
Z
\ar@{->}[r]^-{\delta}
\ar@/^1.5pc/[urr]^-{\nu}
\ar@/_1.5pc/[drr]_-{\gamma}
&
X'
\ar@{->}[dr]_-{g}
\ar@{->}[ur]^-{v}
&\Xi
&Y
\\
&& Y'
\ar@{->}[ur]_-{u}
}
\]
with $\delta$  proper,  $f$ of finite flat dimension,
$u$  flat, and $\Xi$ ~a fiber square.\vspace{1pt}

For $M\in\dcatf f,$ $E\in\dcatf{Y}$ and\/  $N\in\Dqcpl(Y'),$ the following assertions hold.\vspace{1pt}

\textup{(i)}  $u^*\<\<E\Otimes{Y'}\<N\in\Dqcpl(Y').$\vspace{1pt}

\textup{(ii)}  $v^*\RH_{\<X}( M, f^!\<E)\Otimes{X'} \bL g^*\<\<N\in\Dqcpl(X').$\vspace{1pt}

\textup{(iii)} There exist functorial isomorphisms\vspace{-3pt}
$$
\delta^!\big(v^*\RH_{\<X}( M, f^!\<E)\Otimes{X'} \bL g^*\<\<N\big)
\iso\RH_Z\big(\bL\nu^*\<\<M,\> \gamma^!(u^*\<\<E\Otimes{Y'}\< N)\big).
$$
\end{theorem}
\vskip-2pt
(Note  that $v$ is flat, so that $v^*\cong\bL v^*$; and similarly for
$u$.)

\vskip2pt

Before presenting a proof, we derive global
versions of some results established earlier for homomorphisms of
rings.\vspace{-2pt}

\begin{remark}\label{rem:semidualizing, global}
If $\sigma\colon K\to S$ is a homomorphism of rings that is essentially
of finite type
and $g\colon V=\Spec S\to\Spec K=W$ is the corresponding scheme-map
then, with ${}^\sim$
denoting sheafification---an equivalence of categories from $\D(S)$ to 
$\Dqc(V)$, with quasi-inverse $\mathbf R\Gamma(V,-)$,
see \cite[5.5]{BN}---there is an isomorphism in $\D(V)$:\vspace{-2pt}
 \begin{equation}\label{reldual}
g^!\mathcal O_W\simeq (D^\sigma)^\sim.
 \end{equation}
 \vskip-2pt
To see this,  factor $\sigma$ as $K\to P\!:= V^{-1}K[x_1,\dots,x_{\<d}]\twoheadrightarrow S$ (see~\eqref{eq:eft}), so that, correspondingly,
 $g=g^{}_1g^{}_2$ with $g^{}_1$ essentially smooth of relative dimension $d$ 
and $g^{}_2$ a closed immersion; then by ~\eqref{smooth!}, \eqref{finite},
and  Theorem~\ref{thm:uniqueness},\vspace{-2pt}
$$
 g^!\mathcal O_W\simeq g_2^!g_1^!\mathcal O_W\simeq 
 \Shift^{d}\Rhom PS{\Omega^{d}_{P|K}}\simeq (D^\sigma)^\sim.
 $$
 
So the following assertion, for an arbitrary scheme\kern.5pt-map $f\colon X\to Y$, globalizes Theorem \ref{thm:reduction1}(1)---and supports our calling any $\OX$-complex isomorphic in
$\D(X)$ to~$f^!\mathcal O_Y$   a \emph{relative dualizing complex for}~$f$.
Set
$$
D_{\<\<f}\mkern.5mu M\!:=\RH_X(M,\>f^!\mathcal O_Y\<)\qquad\big(M\in\D(X)\big).
$$
Then \emph{the contravariant functor\/ $D_{\<\<f}$ takes\/ $\dcatf f$
into itself, and for every\/ \mbox{$M\in\dcatf f$} the canonical map is an
isomorphism} $M\iso D_{\<\<f}D_{\<\<f}\mkern.5mu M$.

Indeed, the proof of \cite[p.\,259, 4.9.2]{Il3} (in whose first
line (4.8) should be (4.9)) applies here, with ``localizing immersion" in
place of ``open immersion," and with ``essentially smooth''
in place of ``smooth,'' see Remark~\ref{nice}. (Actually, the assertion being
local on both $X$ and $Y\<$, for compactifiability of $f$ one can use
Example~\ref{example} rather than the compactification theorem \cite[4.1]{Nk2}.)
\end{remark}

For\/ $E=\mathcal O_Y$ and $D_{\<\<f}\mkern.5mu M$ in place of $M$,
Theorem \ref{thm:reduction3} and Remark \ref{rem:semidualizing,
global} yield the next Corollary, which bears comparison---at least
formally---with Verdier's ``kernel theorem" \cite[p.\,44, Thm.\,4.1]{V}:
\vspace{-1pt}

\begin{corollary}
\label{cor:reduction3}
\pushQED{\qed}%
Under the assumptions of~\emph{\ref{thm:reduction3}} there exists a natural
iso\-morphism\vspace{-1pt}
\[
\delta^!(v^*\<\<M\Otimes{X'}\bL g^*\<\<N)\iso
\RH_Z(\bL\nu^*\<\<D_{\<\<f}\mkern.5mu M, \gamma^!N)\,.
\qedhere
\]
\end{corollary}

\pagebreak[3]

\begin{corollary}\label{cor:reduction5}
Let $f\colon X\to Y$ be a flat scheme-map. Set $X'\!:=X\times_Y X,$
with canonical projections  $X\xla{\pi_1}X'\xra{\pi_2}X$ and diagonal
map $\delta\colon X\to X'\<$.

There are
natural isomorphisms, for  $M\in\dcatf f,$ $E\in\dcatf{Y}$ and $N\in\Dqcpl(X)\colon$
$$
\delta^!\big(\pi_1^*\RH_{\<X}( M, f^!\<E)\Otimes{X'} \pi_2^*N\big)
\iso\RH_{\<X}\big( M,\>\> f^*\<\<E\Otimes{\<\<X}\< N\big).
$$
\end{corollary}

\begin{proof}
The maps $\pi_1$  and $\pi_2$ are flat along with $f$.  The assertion is
just the special case of Theorem~\ref{thm:reduction3} corresponding to the
data $Z\!:=X$, $Y'\<\!:=X$, $u\!:=f$, $v\!:=\pi_1$, and $g\!:=\pi_2$---so
that $\nu=\gamma=\id^{\<X}$.
 \end{proof}

The first isomorphism in the next corollary is, for flat $f$,  a
globalization of Theorem \ref{thm:reduction2} insofar as the objects
involved are concerned.  This is seen by using the description of
$\delta^!$ given in~\ref{sheaf duality} for the finite map $\delta$,
and the standard equivalence of $\D(S)$ and~ $\Dqc\big(\Spec S\big)$ for
a commutative ring~$S$ \cite[5.5]{BN}.  We won't deal with the relation
between  the corresponding isomorphisms.

\begin{corollary}[Global reduction formulae]\label{formulae}
With $f$ and\/ $\delta\colon X\to X'$ as in ~\ref{cor:reduction5}, there
exist, for\/  $M\in\dcatf f$ 
and\/ $N\in\Dqcpl(X),$ natural iso\-morphisms
\begin{align*}
\delta^!(\pi_1^*M\Otimes{\<\<X'}\pi_2^*N)&\iso \RH_X(\RH_X(M\<,f^!\mathcal O_Y), \>N);
\\
\delta^!\RH_{X'}\<(\pi_1^*M, \pi_2^*N)&\iso \RH_X(M\Otimes{\<\<X}f^!\mathcal O_Y\<, \>N).
\end{align*}
\end{corollary}

\noindent\emph{Proof.}
For  the first isomorphism, apply  \ref{cor:reduction5} with
$E=\mathcal O_Y$ and $D_{\<\<f}\mkern.5mu M$ in place of $M\<$, and
use the isomorphism $M\iso D_{\<\<f}D_{\<\<f}\mkern.5mu M$ from Remark
\ref{rem:semidualizing, global}.

The second isomorphism is the composition 
\begin{align*}
\delta^!\RH_{\<X'}\<(\pi_1^*M, \pi_2^*N)&\underset{\lift1.2,a,}{\iso}
\RH_{\<X}(\bL \delta^*\pi_1^*M,  \delta^!\pi_2^*N)\\
& \underset{\lift1.2,b,}{\iso} \RH_{\<X}(M,  \RH_X(f^!\mathcal O_Y, N))\\
&\underset{\lift1.2,c,}{\iso}
\RH_{\<X}(M\Otimes{\<\<X}f^!\mathcal O_Y\<,\> N),
\end{align*}
\vskip-2pt\noindent
where the isomorphism $a$ comes from \eqref{twisted RHom},
$b$ from the special case \mbox{$M=\OX$} of the first isomorphism in
\ref{formulae}, and $c$ from the first isomorphism in~
\S\ref{adjunction}.

\pagebreak[3]

The following lemma contains the key ingredient for the proof of Theorem
\ref{thm:reduction3}.

\begin{lemma} \label{an iso}
Let\/ $g\colon X'\to Y'$ be a scheme-map of finite flat dimension.
For all $M'\in\dcatf g,$  $E'\in\dcatf {Y'}$ and $F'\in\Dqcpl(Y'),$ the map
from \eqref{Hom and Tensor} is an isomorphism
\begin{equation}
 \label{eq:an iso}
\psi\colon\RH_{\<X'}( M'\<\<,\>\>g^!\<E')\Otimes{X'}\bL g^*\<\<F'\iso
\RH_{\<X'}( M'\<\<,\>\>g^!\<E'\<\Otimes{X'}\bL g^*\<\<F')\,.
\end{equation}
\end{lemma}

\begin{proof}
Using the isomorphisms  \eqref{^* and tensor} and  (for open immersions)
\eqref{^* and Hom}, one checks that everything here commutes with
restriction to open subsets on $X'\<$, whence the question is \mbox{local}
on both ~$X'\<$ and $Y'$ (see Remark~\ref{twisted rem}(b).)
Thus it may be assumed\vspace{1pt} that both $X'$ and~$Y'$ are
affine and that $g$ factors as \mbox{$X'\xrightarrow{\lift.85,i,}
Z'\xrightarrow{\lift.85,h,} Y'$} with $i$ a closed immersion and $h$
essentially smooth.

Since $i_*$~preserves stalks of $\mathcal O_{\<X'}$-modules,
therefore $i_*$ is an exact functor, and furthermore, since $\D$-maps
are isomorphisms if they are so at the homology level, it will suffice
to show that $i_*(\psi)$ $(=\R i_*(\psi))$ is an isomorphism in $\D(Z')$.

Before proceeding, note that
$
\RH_{\<X'}( M'\<\<,\> i^!h^!\<E')\in\Dqcpl(X').
$
That's because $i_*M'\in\Dcmi(Z')$, so the duality isomorphism~\eqref{duality iso} and \cite[p.\,92, 3.3]{H} give
$$
i_*\RH_{\<X'}( M'\<\<,\> i^!h^!\<E')
\cong
\RH_{Z'}( i_*M'\<\<,\> h^!\<E')\in\Dqcpl(Z').
$$
In fact,  $\RH_{Z'}( i_*M'\<\<,\> h^!\<E')$ is \emph{perfect} because $i_*M'$ 
and $h^!\<E'$ are both perfect (see~\eqref{smooth!}, \cite[p.\,130, 4.19.1]{Il1}
and \cite[p.\,148, 7.1]{Il1}). 

Recall from ~\ref{tensor iso}.3 that the map \eqref{Hom and Tensor} is an isomorphism
if the complex $M$ is perfect; and that  the map \eqref{^! and Tensor} is an isomorphism
when $f$ is flat and $E$ is perfect.\vspace{1pt}\looseness=-1

\smallskip

Now, there is the sequence of natural isomorphisms:
\begin{xxalignat}{2}
\quad
i_*\big(\RH_{\<X'}(M'\<\<,\> g^!\<E')\Otimes{X'}&\<\bL g^*\<\<F'\big)
&&{ }
\\
\iso&
i_*\big(\RH_{\<X'}( M'\<\<,\> i^!h^!\<E')\Otimes{X'}\<\bL i^*\bL h^*\<\<F'\big)
&&{ }
\\
\iso&
i_*\RH_{\<X'}( M'\<\<,\> i^!h^!\<E')\Otimes{Z'}\<\bL h^*\<\<F'
&&\text{by \eqref{projection}}\\
\iso&
\RH_{Z'}( i_*M'\<\<,\> h^!\<E')\Otimes{Z'}\<\bL h^*\<\<F'
&&\text{by \eqref{duality iso}}
\\
\iso&
\RH_{Z'}\big( i_*M'\<\<,\> h^!\<E'\Otimes{Z'}\<\bL h^*\<\<F'\big)
&&\text{by \eqref{Hom and Tensor}}
\\
\iso&
\RH_{Z'}\big( i_*M'\<\<,\> h^!(E'\Otimes{Y'}\<F')\big)
&&\text{by \eqref{^! and Tensor}}
\\
\iso&
i_*\RH_{\<X'}\<\big( M'\<\<,\>i^! h^!(E'\Otimes{Y'}\<F')\big)
&&\text{by \eqref{duality iso}}
\\
\iso&
i_*\RH_{\<X'}\<\big( M'\<\<, g^!(E'\Otimes{Y'}\<F')\big)
\\
\iso&
i_*\RH_{\<X'}( M'\<\<,  g^!\<E'\Otimes{X'}\<\bL g^*\<\<F')
&&\text{by \eqref{^! and Tensor}.}
\end{xxalignat}

It can be shown that these isomorphisms compose to $i_*(\psi)$; but we 
avoid this somewhat lengthy verification and instead use a ``way-out" 
argument. Fix $M'$ and~$E'$. Via the above sequence of 
isomorphisms, the source and target of~$i_*(\psi)$, 
considered as functors in~$F'\<$, are isomorphic
to the functor $\Upsilon \colon \Dqcpl(Y') \to \Dqcpl(Z')$ given by 
$\Upsilon(F') = \RH_{Z'}( i_*M'\<\<,\> h^!\<E')\Otimes{Z'}\<\bL h^*\<\<F'$.
Since $\RH_{Z'}( i_*M'\<\<,\> h^!\<E')$ is perfect and $h$ is flat,
it follows that $\Upsilon$ is a \emph{bounded functor} \cite[(1.11.1)]{Lp2}, whence the same 
is true of the source and target of~$i_*\psi$. 

Furthermore,  one checks that $\psi$ (and hence $i_*\psi$) is a morphism of $\Delta$-functors (see \cite[\S1.5]{Lp2}).
By \cite[p.\,69, (iii)]{H}, it 
suffices therefore to prove that $i_*\psi$ \emph{is an isomorphism when $F'$
is a quasi-coherent module.} 

Since $Y'$ is affine, any such  $F'$ is a homomorphic image of a free
$\mathcal O_{\<Y'}$-module. Hence, by \cite[p.\,69, (iii)]{H} (dualized), we may
assume that $F'$ itself is free.

Since $\Upsilon$ respects direct sums in that for any small family 
$(F_\alpha)$ in $\D(Z')$, the natural map is an isomorphism
$$
\oplus_\alpha\Upsilon(F_\alpha)\iso \Upsilon(\oplus_\alpha \>F_\alpha),
$$
the same holds for the source and target of $i_*\psi$. There results a reduction to the trivial case when $F'=\mathcal O_{Y'}$.

This completes the proof of Lemma~\ref{an iso}.
\end{proof}

\begin{proof}[Proof of Theorem \emph{\ref{thm:reduction3}}] Assertion (i) holds because
$u^*\<\<E\in\dcatf{Y'}$.

Since $\Xi$ is  a fiber square, the map $v$ is flat along with $u$.
For the same reason, the map $g$ has finite flat dimension---so that
$\bL g^*\<\<N\in\Dqcpl(X')$, see \cite[\S2.7.6]{Lp2}, and the $\mathcal
O_{\<X'}$-complex $v^*\<\<M$ is $g$-perfect, see \cite[p.\,257, 4.7]{Il3}.
We then have natural isomorphisms
 \begin{align*}
v^*\RH_X( M,\mkern1mu f^!\<E)\Otimes{X'}\bL g^*\<\<N
&\iso
\RH_{\<X'}(v^*\<\<M,\mkern1muv^*\<\<f^!\<E)\Otimes{X'}\bL g^*\<\<N\\
&\,\overset{\beta_{\Xi}\<}\lra\,
\RH_{\<X'}(v^*\<\<M,\mkern1mu g^!u^*\<\<E)\Otimes{X'}\bL g^*\<\<N\\
&\,\overset{\psi}\lra\,
\RH_{\<X'}(v^*\<\<M\<\<,\>\>g^!u^*\<\<E\Otimes{X'}\bL g^*\<\<N)\\
&\iso
\RH_{\<X'}\big(v^*\<\<M, \>\>g^!(u^*\<\< E\Otimes{Y'}\<N)\big)
 \end{align*}
described, respectively, in and around  \eqref{^* and Hom}, \eqref{twisted rem}(b),
\eqref{eq:an iso}, and~\ref{tensor iso}.3.

Since $v^*\<\<M\in\dcatf{g}\subset\Dcmi(X')$ and $g^!(u^*\<\< E\Otimes{Y'}\<N)\in\Dqcpl(X')$, therefore
$$
\RH_{\<X'}\big(v^*\<\<M, \>\>g^!(u^*\<\< E\Otimes{Y'}\<N)\big)\in\Dqcpl(X'),
$$
cf.~\cite[p.\,92, 3.3]{H}. Assertion (ii) in~\ref {thm:reduction3} results.

The composition of the maps  above induces the first isomorphism below:
\begin{align*}
\delta^!\big(v^*\RH_{\<X}( M, f^!\<E)\Otimes{X'} \bL g^*\<\<N\big)
&\iso\,\delta^!\RH_{\<X'}\big(v^*\<\<M,
\mkern1mu g^!(u^*\<\< E\Otimes{Y'}N)\big)\\
&\iso\,
\RH_Z\big(\bL\delta^*v^*\<\<M, \mkern1mu
\delta^!g^!(u^*\<\<E\Otimes{Y'}N)\big)\\
&\iso\,
\RH_Z\big(\bL\nu^*\<\<M,\> \gamma^!(u^{\<*}\<\<E\Otimes{Y'} N)\big).
\end{align*}
The second isomorphism is from \eqref{twisted RHom}.  The third isomorphism is canonical.
\end{proof}

  \section*{Acknowledgments}
At the meeting \emph{Hochschild Cohomology: Structure and applications}
at BIRS in September, 2007, three of the authors discussed with Amnon
Yekutieli and James Zhang possible connections between the results in the
preprint version of \cite{AI:mm}, \cite{YZ1}, and Grothendieck duality.
These conversations started the collaboration that led to the present
paper and the papers \cite{AIL1} and \cite{AIL2}.  We are grateful to
the Banff International Research Station for sponsoring that meeting, to
Yekutieli and Zhang for fruitful discussions, and to Teimuraz Pirashvili
for pointing us to \cite{BP}

\end{document}